\documentclass[final,onefignum,onetabnum]{siamart171218}

\usepackage{lipsum}
\usepackage{amsfonts}
\usepackage{graphicx}
\usepackage{epstopdf}
\usepackage{algorithmic}
\usepackage{comment}
\ifpdf
  \DeclareGraphicsExtensions{.eps,.pdf,.png,.jpg}
\else
  \DeclareGraphicsExtensions{.eps}
\fi

\usepackage{subcaption}
\usepackage{booktabs}
\newsiamremark{remark}{Remark}
\newsiamremark{hypothesis}{Hypothesis}
\crefname{hypothesis}{Hypothesis}{Hypotheses}
\newsiamthm{claim}{Claim}
\newsiamthm{conjecture}{Conjecture}
\newsiamthm{result}{Result}

\usepackage{amsopn}

\makeatletter
\newcommand*{\addFileDependency}[1]{% argument=file name and extension
  \typeout{(#1)}% latexmk will find this if $recorder=0 (however, in that case, it will ignore #1 if it is a .aux or .pdf file etc and it exists! if it doesn't exist, it will appear in the list of dependents regardless)
  \@addtofilelist{#1}% if you want it to appear in \listfiles, not really necessary and latexmk doesn't use this
  \IfFileExists{#1}{}{\typeout{No file #1.}}% latexmk will find this message if #1 doesn't exist (yet)
}
\makeatother

\newcommand*{\myexternaldocument}[1]{%
    \externaldocument{#1}%
    \addFileDependency{#1.tex}%
    \addFileDependency{#1.aux}%
}

\newcommand{\norm}[1]{\left\lVert#1\right\rVert}
\newcommand{\parens}[1]{\left( #1 \right)}

\newcommand{\brackets}[1]{\left[ #1 \right]}
\newcommand{\evalbar}[2]{\left. #1 \right|_{#2}}

\newcommand{\Rbb}{\mathbb{R}}

\newcommand{\Zbb}{\mathbb{Z}}

\newcommand{\Abf}{\mathbf{A}}
\newcommand{\bbf}{\mathbf{b}}
\newcommand{\Bbf}{\mathbf{B}}

\newcommand{\Cbf}{\mathbf{C}}
\newcommand{\dbf}{\mathbf{d}}

\newcommand{\ebf}{\mathbf{e}}
\newcommand{\Ebf}{\mathbf{E}}

\newcommand{\Fbf}{\mathbf{F}}

\newcommand{\Gbf}{\mathbf{G}}
\newcommand{\hbf}{\mathbf{h}}

\newcommand{\Ibf}{\mathbf{I}}

\newcommand{\Kbf}{\mathbf{K}}

\newcommand{\Lbf}{\mathbf{L}}

\newcommand{\nbf}{\mathbf{n}}

\newcommand{\Pbf}{\mathbf{P}}

\newcommand{\Qbf}{\mathbf{Q}}
\newcommand{\rbf}{\mathbf{r}}

\newcommand{\ubf}{\mathbf{u}}
\newcommand{\Ubf}{\mathbf{U}}

\newcommand{\Vbf}{\mathbf{V}}
\newcommand{\wbf}{\mathbf{w}}
\newcommand{\Wbf}{\mathbf{W}}

\newcommand{\dt}{\Delta t}
\newcommand{\Thetabf}{\mathbf{\Theta}}

\newcommand{\Ubfstar}{\mathbf{U}^{\star}}
\newcommand{\Thetabfstar}{\mathbf{\Theta}^{\star}}
\newcommand{\Gbfstar}{\mathbf{G}^{\star}}
\newcommand{\Bbfstar}{\mathbf{B}^{\star}}
\newcommand{\Wbfstar}{\mathbf{W}^{\star}}
\newcommand{\ustar}{u^{\star}}

\newcommand{\Phibf}{\mathbf{\Phi}}
\newcommand{\Phidotbf}{\mathbf{\dot{\Phi}}}

\newcommand{\Psihatbf}{\mathbf{\widehat{\Psi}}}

\newcommand{\vecPhi}{\boldsymbol{\phi}}
\newcommand{\vecPhidot}{\boldsymbol{
\dot{\phi}}}
\newcommand{\mubf}{\boldsymbol{\mu}}

\newcommand{\phidot}{\dot{\phi}}

\newcommand{\flr}[1]{{\lfloor {#1} \rfloor}}
\newcommand{\Mtilde}{\widetilde{M}}

\newcommand{\rchat}{\widehat{r}_c}

\newcommand{\erm}{\mathrm{e}}

\newcommand{\rmin}{r_{\text{min}}}

\newcommand{\ermhat}{\widehat{\mathrm{e}}}

\newcommand{\ebfint}{\mathbf{e}_{\text{int}}}
\newcommand{\ebfhatint}{\mathbf{\widehat{e}}_{\text{int}}}

\newcommand{\what}{\widehat{\mathbf{w}}}
\newcommand{\wstar}{\mathbf{w}^{\star}}

\newcommand{\Kcalbf}{\boldsymbol{\mathcal{K}}}

\usepackage{cite}
\usepackage{nicefrac}
\usepackage{mathrsfs}
\headers{WSciML: Test Function Construction}{A. Tran and D.M. Bortz}

\title{Weak Form Scientific Machine Learning: Test Function Construction for System Identification\thanks{Submitted to the editors DATE.
\funding{This work is supported in part by %National Science Foundation grants 2054085, the
Department of Energy grant DE-SC0023346 and in part by National Institute of General Medical Sciences grant R35GM149335.}}}
\author{April Tran\thanks{Department of Applied Mathematics, University of Colorado, Boulder, CO 80309-0526
(\email{chi.tran@colorado.edu}, \email{david.bortz@colorado.edu}).}
\and David M.~Bortz\footnotemark[2]}

\ifpdf
\hypersetup{
  pdftitle={WSciML: Test Function Support Size},
  pdfauthor={A. Tran and D.M. Bortz}
}
\fi

\myexternaldocument{Tran02_supplement}

\begin{document}

\maketitle

% REQUIRED
\begin{abstract}
Weak form Scientific Machine Learning (WSciML) is a recently developed framework for data-driven modeling and scientific discovery.  It leverages the weak form of equation error residuals to provide enhanced noise robustness in system identification via convolving model equations with test functions, reformulating the problem to avoid direct differentiation of data. The performance, however, relies on wisely choosing a set of compactly supported test functions. %, which shifts derivatives onto these known functions, thus turning the problem into a projection onto a test function basis. As a result, the effectiveness of weak form methods depends heavily on the choice of test functions. 
%While many classes of functions have been proposed, there remains limited guidance for selecting their properties (e.g., support size and smoothness) for optimal system recovery.

In this work, we mathematically motivate a novel data-driven method for constructing \emph{Single-scale-Local} reference functions for creating the set of test functions.  Our approach numerically approximates the integration error introduced by the quadrature and identifies the support size for which the error is minimal, without requiring access to the model parameter values. Through numerical experiments across various models, noise levels, and temporal resolutions, we demonstrate that the selected supports consistently align with regions of minimal parameter estimation error. We also compare the proposed method against the strategy for constructing \emph{Multi-scale-Global} (and orthogonal) test functions introduced in our prior work, demonstrating the improved computational efficiency. 
\end{abstract}

% REQUIRED
\begin{keywords}
  Weak Form, Test Function, WENDy, WSINDy
\end{keywords}

% REQUIRED
\begin{AMS}
  46F05, 62FXX, 62JXX, 65L09, 68Q32
\end{AMS}

\section{Introduction}
\label{sec:intro}
Weak form Scientific Machine Learning (WSciML) is a recently developed framework for data-driven modeling and scientific discovery.  Historically, the weak form of a model equation has been used as a tool to mathematically and computationally study solutions to the equation. Only recently has it been used to solve inverse problems in scientific machine learning.

Towards the goal of improving the performance of WSciML equation learning and parameter estimation, we propose a novel data-driven method for constructing \emph{Single-scale-Local (SL)} test functions. Our approach is based on numerically estimating the integration error introduced by the quadrature. We derive mathematical justifications for criteria to determine, e.g., the support size for which the error is minimal. This is done based entirely on the measured data and without requiring access to the model parameter values. Through numerical experiments across various models, noise levels, and temporal resolutions, we demonstrate that the chosen hyperparameters consistently align with regions of minimal parameter estimation error. We also compare the proposed method against the strategy for constructing \emph{Multi-scale Global (MG)} (and orthogonal) test functions introduced in our prior work, demonstrating the improved computational efficiency. 

\subsection{Problem Setup and Background}
To demonstrate the effectiveness for our approach, we consider the problem of parameter estimation for systems governed by ordinary differential equations (ODEs), where the dynamics take the form
\begin{equation}
\begin{array}{rl}
\dot{u}&=\sum_{j=1}^{J}w_j f_j(u),\label{eq:UPE_DE}\\
u(t_0)&=u_0\in\mathbb{R}^{d},
\end{array}
\end{equation}
with $u:\mathbb{R}\to\mathbb{R}^d$ and $\wbf\in \mathbb{R}^{J}$ denoting an unknown parameter vector. Given time series data $\mathbf{U}\in\mathbb{R}^{(M+1)\times d}$ sampled at discrete time points $t:=\{t_i\}_{i=0}^{M}$, our goal is to estimate $\wbf$. % by minimizing the difference between the model prediction and the data in the least-squares sense:
%\begin{equation}
%\widehat{\wbf}:={\arg \min_{\wbf\in \mathbb{R}^{J}}} %\|u(\mathbf{t};\wbf)-\mathbf{U}\|_{2}^{2}.
%\label{eq:ParEstProblem}
%\end{equation}
Throughout this work, we focus on systems where the right-hand side is a linear combination of known functions $f_j$ (which themselves can be nonlinear). 

Historically, two main classes of methods have been used to solve this parameter estimation problem: Output Error (OE) and Equation Error (EE) approaches \cite{Ljung1999, Ljung2017WileyEncyclopediaofElectricalandElectronicsEngineering}. OE methods propose a candidate parameter vector, numerically solve the ODE, compare (typically via a least squares loss function) the resulting solution to the observed data, and iteratively update the parameters (and the numerical solution) to reduce the loss. This approach, however, introduces several challenges, including unclear impact of solver accuracy, optimization stopping criteria, and general sensitivity to algorithmic hyperparameters \cite{NardiniBortz2019InverseProbl}. 

%$\mathbf{\kappa}$

In contrast, EE methods completely bypass the need for a simulation by substituting noisy observations directly into the model and minimizing the residual of the differential equation. This approach (regression using the differential equation itself) can be traced back to early work in the 1950s and 1960s, particularly in the aerospace literature \cite{Greenberg1951NACATN2340, Lion1967AIAAJournal}, where rapid system identification is crucial for flight control applications. 
More recently, EE-based methods have received renewed attention due to the development of the Sparse Identification of Nonlinear Dynamics (SINDy) framework \cite{BruntonProctorKutz2016ProcNatlAcadSci}, which simultaneously infers model structure and parameter values by identifying a sparse set of active terms from a predefined library of candidate functions. %Unlike the setup in \eqref{eq:ParEstProblem}, SINDy does not assume the functions $f_j$ are known \textit{a priori}.  

To address the challenge of approximating derivatives of (potentially noisy) data in EE methods, Shinbrot was the first to propose integrating the model equations, weighted by compactly supported functions (and applying integration-by-parts) \cite{Shinbrot1954NACATN3288,Shinbrot1957TransAmSocMechEng}. This reformulates the system identification into an algebraic problem. When the model is linear in the parameters, this yields a standard linear regression problem. Later rediscovered and popularized as the Modulating Function Method (MFM) by Loeb and Cahen \cite{LoebCahen1963Automatisme, LoebCahen1965IEEETransAutomControl}, their work led to a wide range of modulating functions. These classes of functions include ones bulit on Fourier-type bases \cite{Shinbrot1957TransAmSocMechEng, PearsonLee1985IEEETransAutomatContr,PearsonLee1985Control-TheoryAdvTechnol,PerdreauvilleGoodson1966JBasicEng}, Hermite interpolating polynomials \cite{Takaya1968IEEETransAutomControl}, Hartley transforms \cite{PatraUnbehauen1995IntJControl}, piecewise polynomials \cite{LiuLaleg-KiratiGibaruEtAl20132013AmControlConf}, and splines \cite{Maletinsky1975IFACProceedingsVolumes,Maletinsky1979IFACProceedingsVolumes}. We note, however, that these previous test functions are defined over an entire domain with the exception of the splines \cite{Maletinsky1975IFACProceedingsVolumes,Maletinsky1979IFACProceedingsVolumes}, which are defined over longer and longer domains. %Among the most commonly used modulating functions in recent literature are asymmetric piecewise polynomials of the form
%\begin{equation*}
%    \psi(t) = t^{\alpha}(t-1)^{\beta},
%\end{equation*}
%which are especially prevalent in system identification problems of the form
%\begin{equation*}
%    \sum_{i = 0}^n a_i(t)y^{(i)}(t) = \sum_{i = 0}^m b_i(t)u^{(i)}(t),
%\end{equation*}
%where $y$ is an observed variable, $u$ is control input, $a_i$ and $b_i$ are unknown coefficients \cite{FedeleColuccio2010ApplMathComput,PinLoveraAssaloneEtAl20132013AmControlConf,LiuLaleg-KiratiPerruquettiEtAl2014IFACProceedingsVolumes},  also fractional order models \cite{LiuLaleg-KiratiGibaruEtAl20132013AmControlConf, LiuLaleg-Kirati2015SignalProcessing,DaiWeiHuEtAl2016Neurocomputing}.
For a comprehensive overview of modulating function development and applications, we refer the reader to \cite{PreisigRippin1993ComputChemEng, PreisigRippin1993ComputChemEnga, PreisigRippin1993ComputChemEngb, Pearson19954654NASAContractorReport, Beier2023}.

Following the introduction of the SINDy framework \cite{BruntonProctorKutz2016ProcNatlAcadSci}, several groups (including ours) independently developed weak formulation-based versions of SINDy \cite{PantazisTsamardinos2019Bioinformatics,WangHuanGarikipati2019ComputMethodsApplMechEng,MessengerBortz2021JComputPhys,MessengerBortz2021MultiscaleModelSimul,GurevichReinboldGrigoriev2019Chaos}.\footnote{We also note that these developments occurred independently from the modulating function literature as well.} In particular, the Weak form Sparse Identification of Nonlinear Dynamics (WSINDy) framework \cite{MessengerBortz2021JComputPhys,MessengerBortz2021MultiscaleModelSimul} uses integration by parts to shift derivatives onto compactly supported \emph{test functions}.\footnote{These \emph{test functions} are so named because of their similarity to the ones used in the Finite Element Method.} This framework has since been applied to a range of settings, including ODEs \cite{MessengerBortz2021MultiscaleModelSimul, Heitzman-BreenDukicBortz2025arXiv250617373}, PDEs \cite{MessengerBortz2021JComputPhys,VaseyMessengerBortzEtAl2025JComputPhys, MinorMessengerDukicEtAl2025arXiv250100738, LyonsDukicBortz2025arXiv250624101}, interacting particle systems \cite{MessengerBortz2022PhysicaD, MessengerWheelerLiuEtAl2022JRSocInterface}, and reduced order modeling \cite{TranHeMessengerEtAl2024ComputMethodsApplMechEng,MessengerBurbyBortz2024SciRep,HeTranBortzEtAl2025IntJNumerMethodsEng}. A broader overview of theoretical developments and applications of our weak form framework appears in \cite{MessengerTranDukicEtAl2024SIAMNews} and an extensive technical introduction in \cite{BortzMessengerTran2024NumericalAnalysisMeetsMachineLearning}.

While WSINDy is effective at learning equations, the regression is a so-called Errors-in-Variables (EiV) problem, which is well-known in the statistics literature \cite{Fuller1987,Buonaccorsi2010}. It is known that measurement noise in EiV problems leads to a bias in the parameter estimates \cite{Regalia1994IEEETransSignalProcess}. Accordingly, to directly address this bias, we considered the simplified problem of \emph{assuming} an existing model and developed the Weak form Estimation of Nonlinear Dynamics (WENDy) method \cite{BortzMessengerDukic2023BullMathBiol} to estimate the parameters. Our approach diverged from the modulating function method in that we both proposed a mathematically-motivated strategy for constructing the test functions as well as a method to correct for the covariance.

However, despite the wide variety of test (modulating) functions proposed in the literature, outside of efforts in our group, there has been little mathematical or empirical guidance for choosing hyperparameters (such as support size or smoothness), particularly in a way that optimizes parameter recovery. In \cite{MessengerBortz2021MultiscaleModelSimul}, we showed that the regularity of test functions directly affects the discretization error of the integrals (see \cref{sec:poly_tf_order} for a summary). In \cite{MessengerBortz2021JComputPhys}, we proposed selecting the support and order of test functions by optimizing their spectral properties to dampen high-frequency noise.  More recently, in \cite{BortzMessengerDukic2023BullMathBiol}, we developed a method for constructing a test function basis that is orthogonal, multiscale, and globally supported over the time domain $[0, T]$ -- referred to here as the \emph{Multiscale Global (MG)} basis -- by analyzing the integration error (see \cref{sec:MinimumRadiusSelection} for details).\footnote{We note that there was an early article \cite{Polubarinova-Kochina1975FluidDyna} which commented on the role of integration error, but subsequent publications did not pursue this line of inquiry.}

In this work, we introduce a new method for selecting the test function basis by minimizing the (estimated) integration error. Crucially, our method only requires access to the data and does not require access to the true parameters.  It also remains robust across a range of noise levels, data resolutions, and regression methods. Specifically, this method focuses on selecting a \emph{Single-scale-Local (SL)} test function basis. Through numerical experiments, we demonstrate that the selected supports consistently align with regions of minimal parameter recovery error. 
We also compare the SL approach to our earlier \emph{Multi-scale-Global (MG)} strategy from \cite{BortzMessengerDukic2023BullMathBiol} in terms of accuracy and walltime.

The remainder of the paper is organized as follows: in \cref{sec:weakform}, we review the mathematical derivation of the WENDy (Weak form parameter Estimation of Nonlinear Dynamics) framework. In \cref{sec:tfselection}, we detail the construction and analysis of the test function set and introduce an algorithm for choosing the test function support size. \Cref{sec:poly_tf_hypeparams} focuses specifically on the piecewise polynomial test function family, where we provide guidance on choosing its hyperparameter. Numerical results are provided in \cref{sec:results}, and concluding remarks are given in \cref{sec:conclusions}.

\subsection{Weak-form Parameter Estimation} 
\label{sec:weakform}
We consider a $d$-dimensional first-order ordinary differential equation (ODE) of the form
\begin{equation}
  \dot{\ubf} = \mathbf{\Theta}(\mathbf{u}) \cdot \mathbf{W}, \quad t\in [0, T],
    \label{eq:ode}
\end{equation}
with the matrix of unknown parameters $\mathbf{W} \in \mathbb{R}^{J \times d} $ and row vectors $\mathbf{u}, \dot{\ubf}$, and $ \mathbf{\Theta}(\mathbf{u})$:
\begin{equation*}
    \begin{aligned}
      \mathbf{u}(t) &:= \left[\begin{array}{ccc}
         u_1(t) &  \cdots &  u_d(t)\\
         \end{array}\right] \in \mathbb{R}^{1 \times d},\\
    \dot{\ubf}(t) &:= \left[\begin{array}{ccc}
         \frac{du_1}{dt}(t) &  \cdots &  \frac{du_d}{dt}(t)\\
         \end{array}\right] \in \mathbb{R}^{1 \times d},\\
    \mathbf{\Theta}(\mathbf{u}) & :=  \left[\begin{array}{ccc}
         f_1(\mathbf{u}) &  \cdots &  f_J(\mathbf{u})\\
         \end{array}\right] \in \mathbb{R}^{{1 \times J}},        
    \end{aligned}
\end{equation*}
and continuous feature functions $f_j: \mathbb{R}^{1 \times d} \rightarrow \mathbb{R}$. To arrive at the weak form, we multiply both sides of \eqref{eq:ode} by a compactly supported test function $\phi(t)$ and integrate by parts to obtain
\begin{equation}
    - \int_0^T \dot{\phi}  \mathbf{u}dt = \int_0^T \phi \mathbf{\Theta}(\mathbf{u}) \mathbf{W}dt.
     \label{eq:weak_ode}
\end{equation}
The goal is to reconstruct the parameter $\mathbf{W}$ from the measurements of $\mathbf{u}$,  observed on a set of discrete time points $\{t_m\}_{m=0}^{M}$  with uniform stepsize of $\Delta t:= \nicefrac{T}{M}$. Eq. \eqref{eq:ode} in the discrete setting then becomes $\dot{\mathbf{U}}= \mathbf{\Theta}(\mathbf{U}) \mathbf{W}$ with matrices\footnote{$\dot{\mathbf{U}}$ is defined similar to $\mathbf{U}$ using $\dot{\mathbf{u}}$. } $ \mathbf{U} \in \mathbb{R}^{(M+1) \times d}$ and $ \mathbf{\Theta(U)} \in \in \mathbb{R}^{(M+1) \times J}$
defined as: 
\begin{equation*}
    \begin{aligned}
        \mathbf{U} := \left[\begin{array}{ccc}
         u_1(t_0) &  \cdots &  u_d(t_0)\\
         \vdots &  \ddots &   \vdots \\
          u_1(t_M) &  \cdots &  u_d(t_M)\\
         \end{array}\right]  \,\,\,\textrm{and}\,\,\,
         \mathbf{\Theta(U)} := \left[\begin{array}{ccc}
         f_1(u(t_0)) &  \cdots &  f_J(u(t_0))\\
         \vdots &  \ddots &   \vdots \\
         f_1(u(t_M)) &  \cdots &  f_J(u(t_M))\\
         \end{array}\right]
    \end{aligned}
\end{equation*}
 We approximate the integrals in Eq. \eqref{eq:weak_ode} using a Newton-Cotes quadrature scheme 
\begin{equation}
    - \boldsymbol{\dot{{\phi}}} \mathbf{U} \approx  \boldsymbol{\phi} \mathbf{\Theta(U)} \mathbf{W},
    \label{eq:discr_phi}
\end{equation}
where the row vectors $\boldsymbol{\phi}, \boldsymbol{\dot{{\phi}}} \in \mathbb{R}^{1 \times (M+1)}$ denote: 
\begin{equation}
\begin{aligned}
    \boldsymbol{\phi}  := \left[\begin{array}{ccc}
         \phi(t_0) &  \cdots &  \phi(t_M)
         \end{array}\right]\mathcal{Q}  \quad\textrm{and}\quad 
     \boldsymbol{\dot{{\phi}}} := \left[\begin{array}{ccc}
         \phidot(t_0) &  \cdots &  \phidot(t_M)
         \end{array}\right]\mathcal{Q}, 
    \end{aligned}
    \label{eq:vecPhi}
\end{equation}
with the quadrature\footnote{In this work, we only use the composite Trapezoidal method to approximate the integral. See \cref{sec:tfselection} below for more information.} matrix $\mathcal{Q} \in \Rbb^{(M+1) \times (M+1)}$.
%In this work, we use the composite trapezoidal rule, with the quadrature weight matrix defined as
%\begin{equation*}
%    \mathcal{Q} = \text{diag}\parens{\frac{\dt}{2}, \dt, \cdots, \dt, \frac{\dt}{2}} \in \Rbb^{(M+1) \times (M+1)}.
%\end{equation*}
%\footnote{Since $\phi_k$ has compact support within $[0, T]$, the composite trapezoidal assigns a uniform weight of $\Delta t$ for each function evaluation. This allows $\dt$ to cancel out on both sides of \eqref{eq:discr_phik}.}:
In practice, we use a collection of $K$ test functions, denoted by $\Kcalbf:= \{ \phi_k\}_{k=1}^K$\footnote{The construction of the set $\Kcalbf$, see \cref{def:tf_set}, and its analysis is presented in  \cref{sec:tfselection}.},
%obtained by translating a reference test function $\psi$ with compact support $[-r, r]$, where the radius $r \leq \frac{T}{2}$. Each $\phi_k$ is centered at a selected time point $t_{m_k}$ from a subset $\{ t_{m_k}\}_{k=1}^K \subset \{t_m\}$. This construction, 
which yields $K$ equations of the form \eqref{eq:discr_phi}, allowing us to cast the identification of $\mathbf{W}$ into a least squares problem:
\begin{equation*}
    \min_{\mathbf{W}} \norm{ \textsf{vec}( \mathbf{B} -\mathbf{G}\mathbf{W}) }_2^2,
    \label{eq:WENDy}
\end{equation*}
where $\mathbf{G: = \mathbf{\Phi}\mathbf{\Theta(U)}} \in \mathbb{R}^{K \times J}$, $ \mathbf{B} := -\mathbf{\dot{\Phi}}\mathbf{U} \in \mathbb{R}^{K \times d}$ and the $\textsf{vec}$ operator vectorizes a matrix. The test function matrices $\Phibf, \Phidotbf \in \Rbb^{K \times (M+1)} $ are defined as
\begin{equation*}
\begin{aligned}
\Phibf   &:=  \left[\begin{array}{ccc} 
\vecPhi_1^T & \cdots & \vecPhi_K^T
 \end{array}\right]^T  \quad\textrm{and}\quad
 \Phidotbf & := \left[\begin{array}{ccc} 
\vecPhidot_1^T & \cdots & \vecPhidot_K^T
 \end{array}\right]^T.
\end{aligned}
\end{equation*}
We formally refer to this framework as the \emph{Weak-form Estimation of Nonlinear Dynamics (WENDy)}. The coefficient matrix $\Wbf$ can  be solved using Ordinary Least Squares \emph{(WENDy-OLS)} via the normal equation 
\begin{equation}
    \Wbf_0 = \parens{\Gbf^T \Gbf}^{-1}\Gbf^T\Bbf.
    \label{eq:w0_OLS}
\end{equation}
%\begin{equation*}
%\begin{aligned}
%\Phibf & :=  \left[\begin{array}{c} 
%\vecPhi_1 \\ \vdots \\ \vecPhi_K \end{array}\right]= 
%\left[\begin{array}{ccc}\phi_1(t_0) &  \cdots &   \phi_1(t_M)\\\vdots &  \ddots &   \vdots \\\phi_K(t_0) &  \cdots &   \phi_K(t_M)\\ \end{array}\right]  \in \Rbb^{K \times (M+1)},\\ \Phidotbf & := \left[\begin{array}{c} \vecPhidot_1 \\ \vdots \\ \vecPhidot_K \end{array}\right]= \left[\begin{array}{ccc}       \dot{\phi}_1(t_0) &  \cdots &   \dot{\phi}_1(t_M)\\        \vdots &  \ddots &   \vdots \\          \dot{\phi}_K(t_0) &  \cdots &   \dot{\phi}_K(t_M)\\        \end{array}\right] \in \Rbb^{K \times (M+1)}.\end{aligned}\end{equation*}

In this work, we also assume that the measurements of $\ubf$ are corrupted by additive Gaussian noise, i.e., $\Ubf_{i,m} = u_i^{\star}(t_m) + \epsilon_i(t_m)$ where $\epsilon_i(t_m)$ is i.i.d. $\mathcal{N}(0, \sigma^2)$. Here, any term with a $\star$ superscript denotes the true (i.e., noise-free) measurement data or is a function of such data, e.g, noise-free data matrix $\Ubfstar$; true parameters $\Wbfstar$; true feature matrix $\Thetabfstar:= \Thetabf(\Ubfstar)$ as well as $\Gbfstar := \Phibf \Thetabfstar$ and $\Bbfstar := -\Phidotbf \Ubfstar$.  

Once measurement noise is introduced, the OLS formulation (e.g., \eqref{eq:w0_OLS}) 
becomes inadequate for identifying the true parameter vector $\Wbfstar$, as it fails to account for the Errors-in-Variables (EiV) nature of the problem, where both the regression matrix $\Gbf$ and response matrix $\Bbf$ are functions of noisy observations. To address this limitation, the \emph{WENDy-IRLS} algorithm \cite{BortzMessengerDukic2023BullMathBiol} employs an \emph{Iteratively Reweighted Least Squares} scheme that approximates the noise-induced covariance structure and yields more reliable parameter estimates. The full details of WENDy-IRLS are provided in \cite{BortzMessengerDukic2023BullMathBiol}, but for the interested reader, the residual decomposition in \cref{sec:ApproxIntegral} combined with the derivation of the covariance correction in \cref{sec:WENDy} provides an overview of the algorithm.

As an aside, we note that in addition to WENDy-OLS and WENDy-IRLS, our group has recently developed WENDy-MLE \cite{RummelMessengerBeckerEtAl2025arXiv250208881}, a maximum likelihood-based version that is capable of parameter estimation for models that are nonlinear in the parameters and subject to an arbitrary (but specified) noise structure. This method is built within the WENDy framework using nonlinear optimization techniques. While we anticipate that the integration error analysis introduced in this paper will still apply to WENDy-MLE, a detailed investigation is beyond the scope of the current work.

%\footnote{We note that we have also recently developed WENDy-MLE, a Maximum Likelihood Estimation-based version that can perform estimation for nonlinear-in-parameters model systems and with arbitrary measurement noise structures. However, this formulation is a nonlinear optimization method, and accordingly, the approach presented in this work will need to be substantially adjusted \cite{RummelMessengerBeckerEtAl2025arXiv250208881}.}

\section{Construction and Analysis of the Test Function Set}
\label{sec:tfselection}
In this section, we describe the construction of the test function set $\Kcalbf$ and present a theoretical analysis of the discretization error arising in weak form integrals involving this set. Building on this analysis, we develop a practical, data-driven method for selecting the test function support radius $r$.

Following our previous work, the set $\Kcalbf$ is constructed by translating a compactly supported reference function $\psi(t)$, as formally defined below.
\begin{definition}[Construction of the Test Function Set $\Kcalbf$]
\label{def:tf_set}
Let $\psi(t)$ be a reference function with compact support on $[-r, r]$, where $r \leq T/2$. Given a temporal grid $\{t_m\}_{m=0}^{M}$ and a set of $K$ center points $\{t_{m_k}\}_{k=1}^K \subset \{t_m\}$ such that $[t_{m_k} - r, t_{m_k} + r] \subset [0, T]$ for each $k$, we construct the test function set $\Kcalbf := \{ \phi_k(t) \}_{k=1}^K$ by
\begin{equation*}
    \phi_k(t) := \psi(t - t_{m_k}), \quad \text{for } k = 1, \dots, K.
\end{equation*}
\end{definition}

In our previous works, we have employed two families of reference test functions $\psi(t;r)$: piecewise polynomial \eqref{eq:tf_poly} and $C^{\infty}$ bump functions \eqref{eq:tf_CinftyBump}. A detailed description of these functions is provided in \cref{sec:reference_tf}.

The test function set $\Kcalbf$ is parameterized by the size of the compact support (e.g., the radius $r$), the number of test functions $K$, and the location of the centers of each $\phi_k$. Optimally choosing all these hyperparameters is challenging, and in this work, we focus our efforts on deriving a mathematical justification for the selection of the radius. Concerning the other properties of $\Kcalbf$, we will adopt strategies based on empirical evidence, e.g., a uniform placement strategy with dense spacing and significant overlap between neighboring test functions\footnote{Increasing the radius $r$ reduces the number of allowable test functions $K$ on a fixed temporal resolution.}. This design choice is motivated by empirical results from \cite{MessengerBortz2021MultiscaleModelSimul}, which demonstrate that overlapping test functions lead to improved parameter recovery\footnote{This setup also enables interpreting the weak-form integrals as convolutions with the reference function $\psi$, which allows efficient evaluation using the Fast Fourier Transform (FFT), as described in \cite{MessengerBortz2021JComputPhys}.}.

With this setup, we now present the main contributions of this paper. We begin in \cref{sec:ApproxIntegral} by establishing an asymptotic expansion for the truncation error arising from discretizing the weak-form equation \eqref{eq:weak_ode} using the trapezoidal rule, as shown in \cref{lem:a}. 
%Then, in \cref{sec:reference_tf}, we describe the two families of reference test functions used in our prior works. Finally, 
From there, in \cref{sec:radiusselection}, we propose a practical method to approximate the integration error without requiring access to the true parameters, and use this approximation to develop a data-driven strategy for selecting the test function radius.

\subsection{Approximating the Weighted Integral}\label{sec:ApproxIntegral}

In the development that follows, the accuracy of the weighted integral approximation defined in \eqref{eq:discr_phi}-\eqref{eq:vecPhi} plays a central role. To understand the motivation for focusing on this approximation, we consider the residual for a $1$-dimensional $\ubf \in \Rbb^{(M+1)}$
\begin{equation*}
\rbf(\mathbf{u},\wbf):=\Gbf\wbf-\bbf \in \Rbb^K,
%\label{eq:WENDyDataResid}
\end{equation*} 
which admits the following decomposition: 
\begin{equation}
    \begin{aligned}
        \rbf(\ubf,\wbf)=\underbrace{(\Gbf-\Gbf^\star)\wbf}_{\begin{array}{c}\ebf_\Theta\end{array}}+\underbrace{\Gbf^\star(\wbf-\wbf^{\star})}_{\begin{array}{c}\rbf_{0}\end{array}}+\underbrace{(\Gbf^\star\wbf^{\star}-\bbf^\star)}_{\begin{array}{c}\ebf_{\text{int}}\end{array}}-\bbf^{\pmb{\varepsilon}}.
    \end{aligned}\label{eq:ResidualDecomposition}
\end{equation}
The terms $\ebf_\Theta$ and $\bbf^\varepsilon$ reflect the impact of the noise on the ODE, while $\rbf_0$ corresponds to the residual with arbitrary $\wbf$ for no measurement noise or integration error. The term $\ebf_{\text{int}}$ is the total numerical integration error induced by the quadrature of the integrals in \eqref{eq:weak_ode} and, as such, is not impacted by the level of noise in the data.  When $\ebf_\textrm{int}$ is negligible, \eqref{eq:ResidualDecomposition} is then dominated by the effect of noise and the specific value of $\wbf$. 

While the choice of the support radius will naturally impact all of the terms in \eqref{eq:ResidualDecomposition}, we focus on its effect on the integration error $\ebfint$. For DEs with bounded solutions (over a compact domain), $\ebf_\textrm{int}(r)$ exhibits an initial exponential decay as $r$ increases, followed by stagnation due to roundoff (see \cref{fig:eehat} for an illustration). 

Based on this observation, we conjecture that the expected error in the parameter estimate is minimized at the radius that minimizes $\ebfint(r)$, which corresponds to the smallest $r$ value beyond which $\ebfint(r)$ no longer improves. However, because the stagnation region has very small magnitude and a nearly flat slope, we avoid optimization, and instead identify the \emph{changepoint $r_c$} in the curve $\log(\norm{\ebfint(r)})$ that marks the transition between exponential decay and stagnation. The logarithmic scale is used to account for the drastic difference in magnitude across the decay range.  In this work, we compute $r_c$ by applying \cref{alg:changepoint} to $\log(\norm{\ebfint(r)})$, over the interval $r \in [t_1, t_{M/2}]$. This algorithm approximates the curve using two linear segments: one fit to the initial decay region and the other to the stagnation region, and selects the changepoint that minimizes the combined fitting error across both intervals\footnote{For additional details, see \cref{sec:changepoint}.}.
We will now state the conjecture formally and defer providing empirical evidence to \cref{sec:OptRad4PWPolyTF} (after we have chosen a specific $\psi$).

%This transition is commonly referred to in the literature as the \emph{knee point} (see, e.g., add paper here), marking a region where further increases in a tunable parameter yield diminishing improvements. While various definitions and algorithms exist to identify such points, we adopt a simple approach tailored to our context: we define the \emph{changepoint} $r_c$ as the value of $r$ for which both the exponential decay region and the stagnation region of $\log(\norm{\ebfint(r)})$ are each best fit by a line:
%\begin{definition}
% Let $r_c$ denote the point at which the monotonic decrease of $\log(\norm{\ebfint(r)})$ stagnates, i.e.,
 %   \[    r_c:=\arg \min_r \left(\left|\frac{\log(\norm{\ebfint(r)})-\log(\norm{\ebfint(t_1)})}{r-t_1}\right|+\left|\frac{\log(\norm{\ebfint(\nicefrac{t_M}{2})})-\log(\norm{\ebfint(r)})}{(\nicefrac{t_M}{2})-r}\right|\right)    \]
%\end{definition}

%Based on this observation, we conjecture that the expected error in the parameter estimate is minimized at the radius that minimizes $\ebfint(r)$, which corresponds to the smallest $r$ value beyond which $\ebfint(r)$ no longer improves. However, because the stagnation region has very small magnitude and a nearly flat slope, we opt for change point detection of $\log(\norm{\ebfint})$ rather than optimization. We will now state this formally and defer providing empirical evidence to \cref{sec:OptRad4PWPolyTF} (after we have chosen a specific $\psi$). 

\begin{conjecture}
Assume that residual in \eqref{eq:ResidualDecomposition} is zero-mean, i.e., $\mathbb{E}[\rbf(\ubf, \wbf)] = \mathbf{0}$. For $r_c$ denoting the point at which the monotonic decrease of $\log(\norm{\ebfint(r)})$ stagnates, we conjecture that 
$$ r_c \equiv \arg\min_{r} \mathbb{E}\brackets{\norm{\widehat{\wbf}(r) - \wstar}_2}, $$
where $\widehat{\wbf}(r)$ is the WENDy estimator with test function radius $r$.
\label{conj:optimal_radius}
\end{conjecture}

Guided by this conjecture, our efforts will now focus on accurately estimating the error in the quadrature.  For reasons that will become clear below, we use the composite trapezoidal rule, leading to the quadrature weight matrix
\begin{equation*}
\mathcal{Q} = \text{diag}\parens{\frac{\dt}{2}, \dt, \cdots, \dt, \frac{\dt}{2}} \in \Rbb^{(M+1) \times (M+1)}.
\end{equation*}
We will now closely consider the truncation error introduced when discretizing the weak-form equation \eqref{eq:weak_ode} for a one-dimensional solution $u$. The following lemma establishes an asymptotic expansion of this error for a test function $\phi$.
\begin{lemma}
    \label{lem:a}
    Let $\ustar(t)$ be a smooth solution to the $1$-dimensional ODE in \eqref{eq:ode}, and let $\ebf_{\text{int}}(\phi, \ustar):= \vecPhi \Thetabf(\ubf^{\star})\wbf + \vecPhidot \ubf^{\star}$ denote the truncation error from applying the trapezoidal rule to the weak form integrals in \eqref{eq:weak_ode}. Then, the following asymptotic expansion holds: 
  \begin{equation*}
      \ebf_{\text{int}}(\phi, \ustar) = \frac{1}{\sqrt{T}} \sum_{n \in\Zbb} \widehat{\phi}_n I_n,
      %\label{eq:e_int}
  \end{equation*}
  where $\widehat{\phi}_n$ are the Fourier coefficients of the test function $\phi(t)$, and $I_n$ is given asymptotically\footnote{The ``$\sim$'' notation denotes asymptotic equivalence, meaning that truncating the series provides increasingly accurate approximations, but the series may not converge.} by
  \begin{equation}
    \begin{aligned}
        I_n & \sim \ustar \big|_{0}^{T}  + \sum_{s=1}^{\infty} \dt^{2s} \frac{B_{2s}}{(2s)!} 
          \parens{ \sum_{l = 0}^{2s} \binom{2s}{l} 
         \parens{ \frac{2\pi in}{T} }^{2s-l}u^{\star(l)}\big|_{0}^{T}},
    \end{aligned}
    \label{eq:I_n}
\end{equation}
and $B_{2s}$ are the non-zero Bernoulli numbers.
\end{lemma}
\begin{proof}
This proof follows from the Euler–Maclaurin formula applied to Fourier integrals of a smooth solution $u(t)$. For further details, see 
\cref{sec:proof_lemma}.
\end{proof}

\subsection{Radius Selection}
\label{sec:radiusselection}
In this subsection, we present the primary contribution of this work: a robust, efficient, and practical method for choosing the radius of the compact support for the reference test function.  This method leverages an approximation for the integration error that arises from discretizing weak-form integrals using the test function set 
$\Kcalbf$. We build on the asymptotic expansion established in \cref{lem:a} to  enable a data-driven strategy for selecting the radius $r$ that minimizes integration error without requiring knowledge of the true parameter $\Wbfstar$.

Recall that for a one-dimensional solution $\ustar$, the discretization error introduced by applying the trapezoidal rule to the weak-form integrals is given by
\begin{equation}
    \ebf_{\text{int}}(\ubf^{\star}) := \Gbfstar \wbf^{\star} - \bbf^{\star} = \Phibf \Thetabfstar \wbf^{\star} + \Phidotbf \ubf^{\star}  \in \Rbb^K.
    \label{eq:e_int}
\end{equation}
 Each component of $\ebfint$ corresponds to the integration error associated with a test function $\phi_k$, where the set $\Kcalbf$ is defined by translating the reference test function $\psi$ as described in \cref{def:tf_set}. 

Motivated by \cref{lem:a}, which provides an asymptotic expansion associated with a single test function $\phi$, we extend this formulation to approximate the full vector $\ebfint$. To do so, we now introduce several components that appear in the approximation. Let $\Pbf_{(k)} \in \Rbb^{K \times M}$ be a \emph{row selection matrix} that extracts rows indexed by test function centers $\{ t_{m_k}\}$, such that $\Pbf_{(k)} \Abf$ returns the subset of rows $\Abf_{m_k}$ for any $\Abf \in \Rbb^{M \times M}$. $\Pbf_{(k)}$ is defined entrywise by
\begin{equation}
\label{eq:P_proj}
 \left[\Pbf_{(k)}\right]_{i, j} := \begin{cases} 
     &  1 \quad i = k \text{ and } j = m_k, \\
    & 0 \quad \text{otherwise}.
        \end{cases}
\end{equation} 
We define the set of M frequencies\footnote{Note that $n$ here is also used as a zero-based index, with negative frequencies wrap around as upper indices.} as 
\begin{equation}
\label{eq:freq}
    \begin{aligned}
        \nbf = \begin{cases}
     \{ 0, 1, \cdots, \frac{M}{2}-1, -\frac{M}{2}, \cdots, -1 \}  &\quad M \text{ even }, \\
      \{ 0, 1, \cdots, \frac{M-1}{2}, -\frac{M-1}{2}, \cdots, -1 \} &\quad M \text{ odd }. \\
\end{cases}
\end{aligned}
\end{equation}
From this frequency set, we construct a diagonal matrix $\Psihatbf \in \mathbb{C}^{M \times M}$, whose entries are the Fourier coefficients of the reference function $\psi$, along with a vector $\dbf \in \mathbb{C}^M$, given by
\begin{equation}
\begin{aligned}
    \Psihatbf := \text{diag} \parens{\parens{\widehat{\psi}_n}_{n \in \nbf}}, \quad \dbf := \parens{\frac{2\pi i n}{T}}_{n \in \nbf}
\end{aligned}
\label{eq:psihat_d}
\end{equation}
Each integral $I_n$ (as defined in \eqref{eq:eint_expand}) appearing in the expansion from \cref{lem:a} is approximated using a truncated Euler–Maclaurin series of order $S \in \mathbb{N}$. The required endpoint derivatives of $\ustar$ are approximated via finite difference schemes, where the $l$-th derivative is estimated using a stencil of order $\mu_l$, denoted $\mathcal{D}^l_{\mu_l}[\cdot]$. The collection of these stencil orders is stored in the vector $\boldsymbol{\mu} = [\mu_1, \dots, \mu_{2S}]$. 

With all components in place, we now present the approximation for $\ebfint$.%, denoted $\ebfhatint$, in \cref{prop:eint_hat}.
\begin{proposition}
\label{prop:eint_hat}
Let $\ustar(t)$  be a smooth solution to the one-dimensional ODE in \eqref{eq:ode}, and consider a collection of $K$  test functions $ \{ \phi_k \}_{k=1}^K$ defined in \cref{def:tf_set} with centers $\{t_{m_k}\}_{k=1}^K$, where $t_{m_k} = m_k \dt$. The integration error $\ebf_{\text{int}} = \Gbfstar \wbf^{\star} - \bbf^{\star} \in \Rbb^{K}$ \eqref{eq:e_int} admits the following approximation:
\begin{equation}
\widehat{\ebf}_{\mathrm{int}}(\ubf^\star, S, \boldsymbol{\mu}) := \frac{1}{\sqrt{T}}\Pbf \Fbf \Psihatbf \Ibf(\ubf^\star, S, \boldsymbol{\mu}) \in \Rbb^K,
\label{eq:eint_hat}
\end{equation} 
where $\Pbf \in \Rbb^{K \times M}$ and $\Psihatbf \in \mathbb{C}^{M \times M}$ are defined in \eqref{eq:P_proj}  and \eqref{eq:psihat_d} respectively; $\Fbf \in \mathbb{C}^{M \times M}$ is the Discrete Fourier Transform (DFT) matrix, and the vector\footnote{Throughout this paper, we use bold uppercase letters for matrices and bold lowercase letters for vectors. However, for the computation of $\ebfhatint$, we use $\Ibf$ to denote a vector.}
 $\Ibf (\ubf^\star, S, \boldsymbol{\mu}) \in \mathbb{C}^M$ is given by
\begin{equation}
    \begin{aligned}
        \Ibf(\ubf^\star, S, \boldsymbol{\mu}) := \ustar|_0^T + \sum_{s=1}^{S} \dt^{2s} \frac{B_{2s}}{(2s)!} \parens{ \sum_{l = 0}^{2s} \binom{2s}{l} \dbf^{2s-l} \mathcal{D}^{l}_{\mu_l}[\ubf^\star]|_0^T },
    \end{aligned}
    \label{eq:I_approx}
    \end{equation}
with $\dbf \in \mathbb{C}^M$ defined in \eqref{eq:psihat_d}.  
\end{proposition}

\begin{proof}
The proof follows by applying \cref{lem:a} to each test function $\phi_k \in \Kcalbf$. Full details are provided in \cref{sec:proof_prop_ehat}.
\end{proof}

\begin{remark}
\label{remark:FFT}
The approximation in \eqref{eq:eint_hat} can be computed efficiently using the Fast Fourier Transform (FFT). In particular,  
$$\widehat{\ebf}_{\mathrm{int}} = \frac{1}{\sqrt{T}}\Pbf \Fbf \Psihatbf \Ibf = \Pbf\, \mathcal{F}[\Psihatbf \Ibf], $$ 
where $\mathcal{F}[\cdot]$ denotes the discrete Fourier transform, efficiently evaluated via FFT. Since both $\Pbf$ and $\Psihatbf$ are sparse, their corresponding matrix operations can also be performed with minimal computational overhead.
\end{remark}

%\subsection{Radius Selection}
%\label{sec:radiusselection}

Guided by \cref{conj:optimal_radius}, we can now state our data-driven strategy for selecting the radius $r$.% that yields low integration error.%, using only the observed data and without having access to the true parameter $\Wbfstar$.

\vspace{2mm}
\noindent\fbox{\begin{minipage}{0.95\textwidth}
\begin{definition}
For a smooth solution $\ustar$ to the $1$-dimensional ODE in \eqref{eq:ode}, we identify the \underline{critical radius $\rchat$} as the changepoint of $log( \widehat{\erm}(\ubf^\star, r))$, where
\begin{equation}
\label{eq:ehat_u}
    \begin{aligned}
        \widehat{\erm}(\ubf^\star, r) := \frac{1}{\sqrt{K(r)}}\norm{\widehat{\ebf}_{\text{int}}(\ubf^\star, r, S, \boldsymbol{\mu})}_2
         =  \frac{1}{\sqrt{K(r)}}\norm{\frac{1}{\sqrt{T}}\Pbf(r)\Fbf \Psihatbf(r) \Ibf(\ubf^\star, S, \boldsymbol{\mu}) }_2,
    \end{aligned}
\end{equation}
for $r = m\dt$ for $m = 2, 3, \cdots, \lfloor M/2 \rfloor$, and $\ebfhatint(\ubf^\star, r, S, \boldsymbol{\mu}) \in \Rbb^{K(r)}$ is defined in  \cref{prop:eint_hat}.
\end{definition}
\end{minipage}}
\vspace{2mm}

We note that the number of test functions $K(r)$, their centers, and the structure of the selection matrix $\Pbf(r)$ all depend on the radius $r$. We also note that the $\log$ transform does not change the $\rchat$ value, but allows a changepoint algorithm to effectively find $\rchat$ as the transition point between two linear functions.

This formula has broad implications for WSciML techniques, and here we present several important remarks concerning implementation and implications:

\begin{remark}
To extend this strategy to a $d$-dimensional system $\Ubfstar$,
%we redefine the regression matrix $\Gbf$ and define the response vector $\bbf$ to incorporate all $d$ dimensions of the observed data $\Ubf$
%\begin{equation*}
%    \begin{aligned}       \Gbf := [\mathbb{I}_{d}\otimes\Phibf\Thetabf] \in \Rbb^{Kd \times Jd}, \quad\bbf :=-\mathsf{vec}(\Phidotbf\Ubf) \in \Rbb^{Kd},\quad\wbf := \mathsf{vec}(\Wbf) \in \Rbb^{Jd},
%    \end{aligned}
%\end{equation*}
%where $\otimes$ denotes the Kronecker product; this allows the integration error to be succinctly written as $\Gbf^\star\wstar-\bbf^\star$. 
we can define the overall estimated integration error by concatenating the vectors $\widehat{\ebf}_{\text{int}}(\Ubfstar_i, r) \in \Rbb^K $ for each component $\Ubfstar_i$ (omitting the arguments $S$ and $\boldsymbol{\mu}$ for clarity). The critical radius $\rchat$ is then identified as the changepoint of the curve $\log\big( \widehat{\erm}(\Ubfstar, r) \big)$, where
\begin{equation}
\label{eq:ehat_U}
    \begin{aligned}
        \widehat{\erm}(\Ubfstar, r)& := \frac{1}{\sqrt{K(r)}}\norm{ 
        \left[\begin{array}{cccc}
    \widehat{\ebf}_{\text{int}}(\Ubfstar_1, r)^T & \widehat{\ebf}_{\text{int}}(\Ubfstar_2, r)^T & \cdots & \widehat{\ebf}_{\text{int}}(\Ubfstar_d, r)^T
    \end{array}\right]^T }_2.
    \end{aligned}
\end{equation}
%and each $\widehat{\ebf}_\textrm{int}(\Ubfstar_i,r)\in \Rbb^{Kd}$.
\end{remark}

\begin{remark}While the procedure constitutes a sweep over all values of radius $r$,  the computation of $\widehat{\erm}(\Ubfstar, r)$ is highly parallelizable, as each radius value can be processed independently. Furthermore, the vectors $\Ibf(\Ubfstar_i, S, \boldsymbol{\mu})$ \eqref{eq:I_approx}, discussed in \cref{prop:eint_hat}, are independent of $r$ and may therefore be computed once and reused across all evaluations. The construction of  $\widehat{\ebf}_{\text{int}}(\Ubfstar_i, r,  S, \boldsymbol{\mu})$ can also be efficiently implemented via the Fast Fourier Transforms, as discussed in \cref{remark:FFT}.
\end{remark}

\begin{remark}In practice, only noisy data $\Ubf$ is available. Nonetheless, the same procedure can be applied by using $\Ubf$ to compute $\widehat{\erm}(\Ubf, r, S, \boldsymbol{\mu})$. This requires approximating derivatives of $\Ubf$ at the endpoints up to order $2S$, which can be challenging. However, since derivatives are needed only at the boundaries, a relatively small number of points (much fewer than $M+1$) suffices. Consequently, error accumulation across the full domain is avoided, and the noise amplification typically associated with numerical differentiation is substantially mitigated.
\end{remark}

\begin{remark}The truncation order $S$ should ideally reflect the regularity of the data (see \cref{sec:apdxS} for guidelines); however, in all cases considered here, $S=1$ was sufficient. For an $l$-th order derivative, a higher finite difference order $\mu_l$ requires more points - specifically, $\mu_l + l$ points - which increases variance. To balance accuracy and robustness, we set $\mu_l = 2S - l + 1$, resulting in a total of $2d(2S+1)$ noisy data points involved in finite difference approximations. 
\end{remark}

\subsection{Demonstration of Optimal Radius for Piecewise Polynomial Test Functions}\label{sec:OptRad4PWPolyTF}

We now demonstrate the effectiveness of the proposed error estimator from \cref{prop:eint_hat} by comparing the estimated integration error to the true error on two systems. To establish a ground truth benchmark, we define the true integration error norm as
\begin{equation}
\erm(r) := \frac{1}{\sqrt{K}}\norm{\ebf_{\text{int}}(r)}_{F} = \frac{1}{\sqrt{K}}\norm{\Gbfstar(r) \Wbfstar - \Bbfstar(r)}_{F},
\label{eq:e}
\end{equation}
where $\ebf_{\text{int}}(r)$ is the truncation error incurred by applying the trapezoidal rule to the weak-form integrals.

\begin{figure}[htbp]
  \centering
  \begin{subfigure}[b]{0.49\textwidth}
    \includegraphics[width=\textwidth]{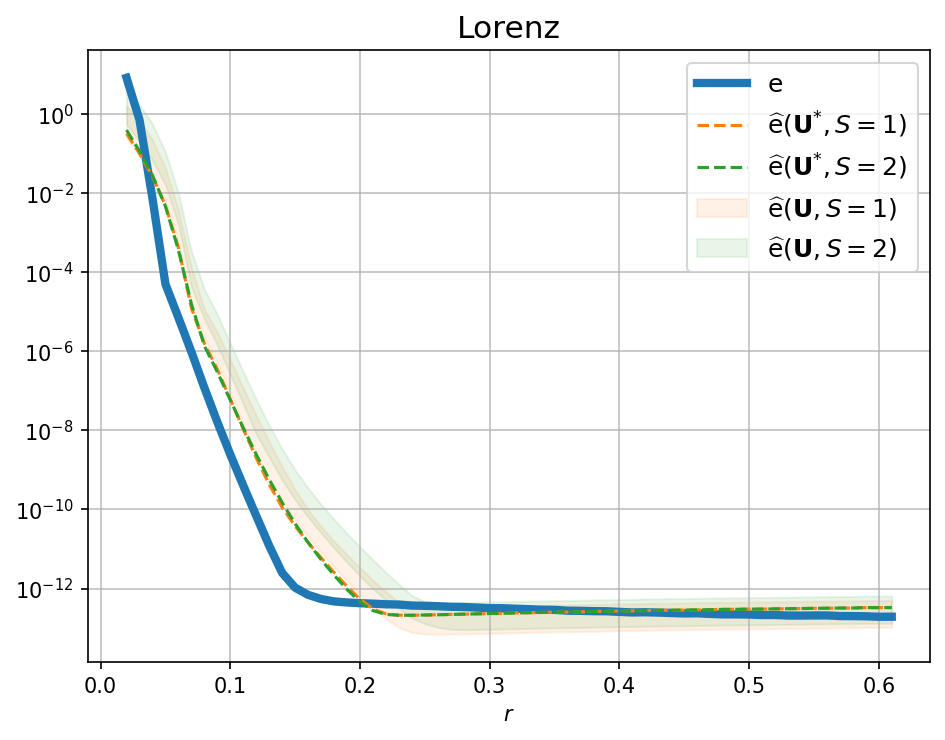}
  \end{subfigure}
  \hfill
  \begin{subfigure}[b]{0.49\textwidth}
    \includegraphics[width=\textwidth]{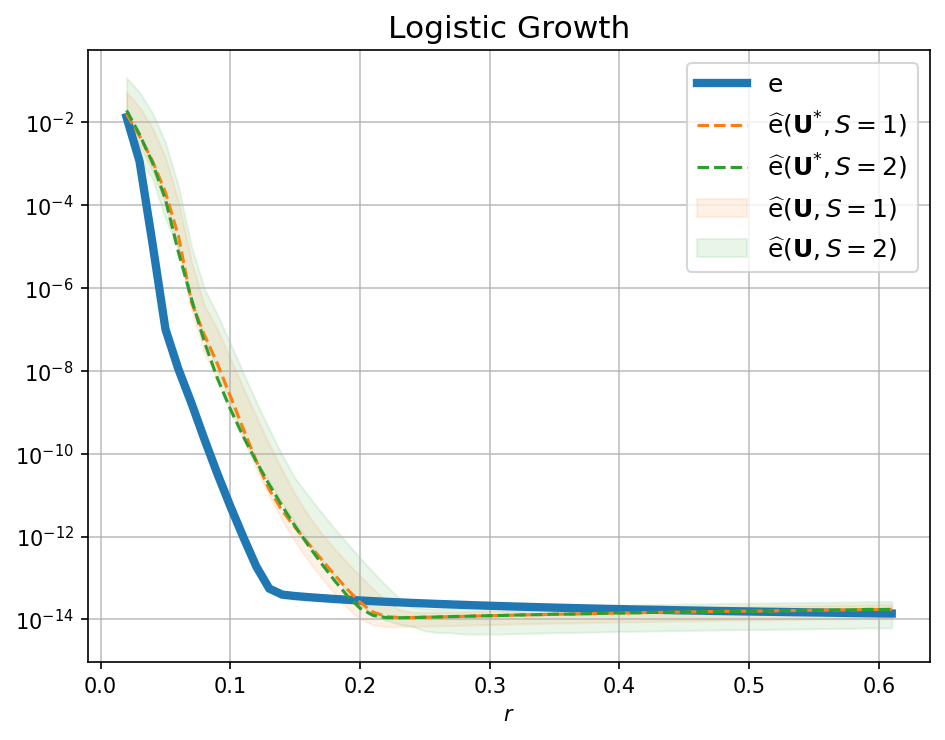}
  \end{subfigure}
  \caption{Comparison of the true integration error $\erm$ \eqref{eq:e} and the approximated error $\widehat{\erm}$ \eqref{eq:ehat_U} for the Lorenz system (left) and logistic growth model (right), with $T= 10$ and $\dt = 0.01$ for both cases. Each plot includes results from noise-free data as well as the range of estimates from $50$ realizations with $20\%$ noise.  For both cases, piecewise polynomial test functions with order $p=16$ are used. The approximated error $\widehat{\erm}$ is computed using truncation orders $S = 1$ (orange) and $S=2$ (green). In both plots, the solid blue line represents the true integration error $\erm$, dashed lines show $\widehat{\erm}$ computed from noise-free data $\Ubfstar$, and the shaded region captures the range of $\ermhat$ across realizations of noisy $\Ubf$. Across both examples, $\widehat{\erm}$ is insensitive to the choice of $S$ and $\widehat{\erm}(\Ubfstar)$ lies closely to $\widehat{\erm}(\Ubf)$. Notably, the changepoint in $\widehat{\erm}(r)$ consistently appears to the right of the changepoint in $\erm(r)$, where the true integration error remains small. This suggests that choosing the radius based on the change point of $\widehat{\erm}$ reliably yields low integration error.}
  \label{fig:eehat}
\end{figure}

\Cref{fig:eehat} shows the true integration error $\erm(r)$ and the estimated error $\widehat{\erm}(r)$ as functions of the test function radius $r$, plotted on a logarithmic scale. For this demonstration, we use the piecewise polynomial test functions as defined in \eqref{eq:tf_poly_main}. As expected, the true error $\erm(r)$ decreases with increasing $r$ and levels off after some radius. In practice, estimating this behavior using only noisy data is nontrivial, especially without access to the true parameters or noise-free observations. Nevertheless, across both the Lorenz and logistic growth systems, \cref{fig:eehat} illustrates that the estimated error $\widehat{\erm}(r)$ remains stable under different truncation orders $S$ and in the presence of noise. In all cases, the changepoint of $\widehat{\erm}(r)$ consistently lies within a region where the true integration error $\erm(r)$ is already small.

%Our method successfully identifies a region of radii over which the integration error $\ebfint(r) = \Gbfstar(r)\wstar - \bbf^\star(r)$ remains small. A natural question arises: can we safely select any radius within this low-error region and expect accurate parameter recovery, i.e., small $\norm{\wbf - \wstar}$? The answer is \emph{not necessarily}.  \Cref{fig:params_error} plots the mean relative parameter error (as defined in \eqref{eq:metric}) for the Lorenz system, with and without noise, as a function of the radius. These empirical results reveal that the parameter error depends both on the integration error and the error induced by measurement noise. In the noise-free setting, only integration error is present, parameter accuracy improves as we enter the low-error region and remains near machine precision. In contrast, in the noisy settings, integration error becomes less significant as more noise is added, and increasing the radius beyond the optimal range leads to a rise in parameter error due to noise amplification.

\cref{conj:optimal_radius} is supported by \cref{fig:params_error}, which shows that the observed minima in parameter error (scatter plot markers) lie close to the detected change point of $\erm(r)$ (blue vertical dashed line). One could attempt to formalize this by minimizing a proxy for mean-squared parameter error, 
$$\text{MSE}(\wbf, \wstar) = \mathbb{E}\brackets{\norm{\wbf - \wstar}^2} = \parens{\text{Bias}(\wbf, \wstar)}^2 + \text{Tr}(\text{Cov}(\wbf)),$$
via an optimization or grid search formulation, however, this would generally require access to ground truth parameter $\wstar$ or noise-free data $\ustar$, which are unavailable and beyond the current scope.
For now, we adopt this conjecture as a practical criterion since the change points of $\erm(r)$ and $\ermhat(r)$, are close (as shown in \cref{fig:eehat}). We choose $\rchat$ (defined as the change point of $\ermhat(r)$) as the reference test function radius. and construct the test function set $\Kcalbf$ according to \cref{def:tf_set}. A formal justification of the conjecture is left for future work.

\begin{figure}[htbp]
  \centering
\includegraphics[width=\textwidth]{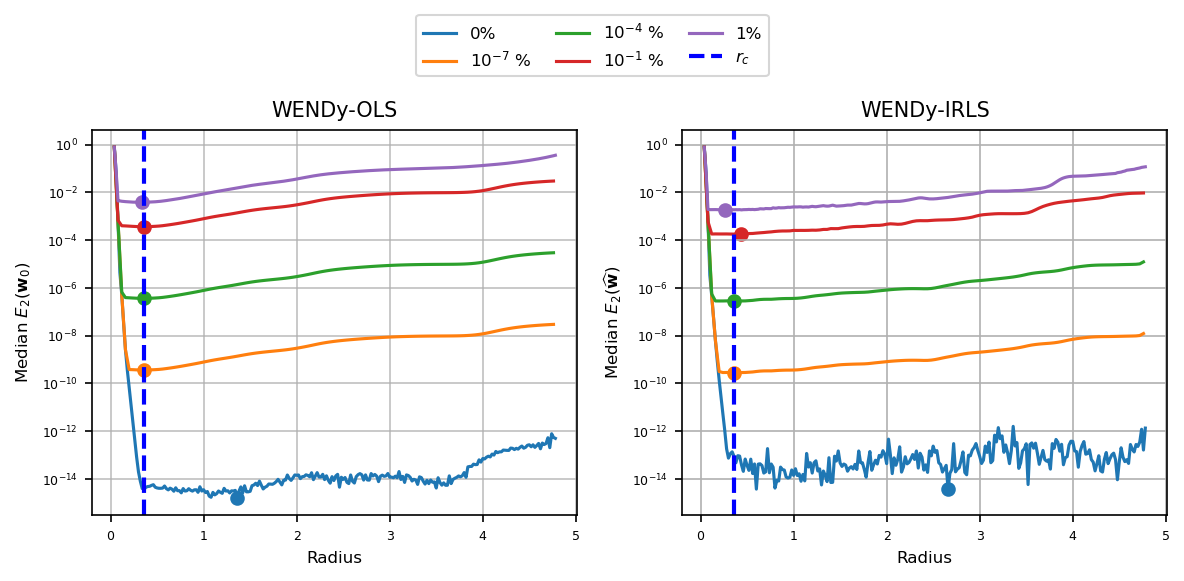}
  \caption{Mean parameter errors for the Lorenz system as a function of radius under WENDy-OLS (left) and WENDy-IRLS (right). Piecewise polynomial test functions with $p=16$ are used, with temporal resolution $M=500$. Each scatter point marks the minimum of the corresponding curve, while the blue dashed vertical line indicates the detected change point $r_c$ of $\erm(r)$. The minima align closely with the identified change point.}
\label{fig:params_error}
\end{figure}

We conclude this section by summarizing the proposed method in \cref{alg:SL} to compute the critical radius $\rchat$ and construct the test function matrices $\Phibf$ and $\Phidotbf$. This algorithm takes as input a reference test function $\psi$ with a fixed secondary parameter
%\footnote{See \cref{sec:poly_tf_hypeparams} for details on available reference test functions.}, along with its associated shape parameter 
(either the order 
$p$ for the piecewise polynomial or the shape parameter $\eta$ for the $C^\infty$ bump functions). The user must also specify the truncation order $S$, the finite difference order $\mubf$, and details about test function centers. This construction strategy is referred to as the \emph{Single-scale-Local (SL)} method, in contrast to the \emph{Multi-scale-Global (MG)} approach described in \cref{sec:MinimumRadiusSelection}. The full procedure is summarized in \cref{alg:SL}.

\begin{algorithm}
\caption{Single-scale-Local Test function set construction}
\label{alg:SL}
\begin{algorithmic}
\STATE \textbf{Input:} Data $\Ubf$, Reference test function $\psi$, Truncation order $S$, Finite Difference orders $\boldsymbol{\mu}$
\STATE \textbf{Output:} Test function matrices $\Phibf$, $\Phidotbf$
\vspace{0.5em}
\STATE Precompute the Fourier integral approximation $\Ibf(\Ubf_i, S, \boldsymbol{\mu})$ for each variable $i = 1, \dots, d$. %using:
%$$\Ibf := \Ubf_i|_0^T + \sum_{s=1}^{S} \dt^{2s} \frac{B_{2s}}{(2s)!} \sum_{l=0}^{2s} \binom{2s}{l} \dbf^{2s-l} \mathcal{D}^{l}_{\mu_l}[\Ubf_i]|_0^T.$$
%\vspace{0.5em}
\STATE For each grid length $m = 2, \dots, \lfloor M/2 \rfloor$ corresponding to candidate radius $r = m \dt$:
\begin{itemize}
  \item Construct $\Psihatbf(r)$.
  \item For each $\Ubf_i$, compute  $\widehat{\ebf}_{\text{int}}(\Ubf_i, r) = \Pbf \mathcal{F}\brackets{\Psihatbf(r) \Ibf(\Ubf_i, S, \mubf )}$ such that 
  $$
\Ibf(\Ubf_i, S, \mubf ) = \Ubf_i|_0^T + \sum_{s=1}^{S} \dt^{2s} \frac{B_{2s}}{(2s)!} \sum_{l=0}^{2s} \binom{2s}{l} \dbf^{2s-l} \mathcal{D}^{l}_{\mu_l}[\Ubf_i]|_0^T.
$$
Concatenate over $i$ to form  $\widehat{\ebf}_{\text{int}}(r)$.
\item Compute $\widehat{\erm}(r) := \frac{1}{\sqrt{K(r)}} \norm{\widehat{\ebf}_{\text{int}}(r) }_2$.
\end{itemize}
\vspace{0.5em}
%\STATE Identify the critical radius $\rchat := m^\star \dt$ where $m^\star := \text{change point}(\widehat{\ebf}(m\dt))$
\STATE Identify the critical radius $\rchat := \text{changepoint}(log(\widehat{\erm}(r))$ using \cref{alg:changepoint}.
\RETURN $\Phibf(\rchat), \Phidotbf(\rchat)$
\end{algorithmic}
\end{algorithm}

\section{Piecewise Polynomial Hyperparameters Selection}
\label{sec:poly_tf_hypeparams}

Throughout this paper, we adopt the piecewise polynomial reference\footnote{For more details about the reference test functions, see \cref{sec:reference_tf}.} test function defined by 
\begin{equation}
  \psi(t;r, p) =
\begin{cases}
     C(r-t)^p(r+t)^p \quad & t \in [-r, r],\\
     \quad  \quad 0 \quad & \text{otherwise},
     \label{eq:tf_poly_main}
\end{cases}
\end{equation} 
with an $L_2$ normalization\footnote{This normalization mirrors that of the $C^\infty$ reference function to ensure compatibility with the minimum radius selection algorithm in \cref{alg:MG}.}, i.e., $\norm{\psi(t)}_{L_2} = 1$. This choice is motivated by the fact that the Fourier transform of $\psi$ admits a closed-form expression, as shown in \cref{prop:psi_hat} below. Having an analytical formula for the Fourier coefficients enables a more accurate integration error approximation $\widehat{\ebf}_{\mathrm{int}}(\ustar, S, \boldsymbol{\mu}) := \frac{1}{\sqrt{T}} \Pbf \Fbf \Psihatbf \Ibf(\ustar, S, \boldsymbol{\mu})$, discussed in \cref{prop:eint_hat}, without introducing additional discretization error from constructing the diagonal matrix $\Psihatbf$. 

%For this reason, we focus on the piecewise polynomial test function family throughout the remainder of the paper. We begin with a brief recap of the radius selection procedure, accompanied by the closed-form expression for the Fourier coefficients of $\psi$ in \cref{prop:psi_hat}. We then turn to an analysis of the polynomial order 
%in \cref{sec:poly_tf_order} and provide general guidelines for selecting this parameter.
%For this reason, we focus on the piecewise polynomial test function family throughout the remainder of the paper. We begin with a discussion of polynomial order selection in \cref{sec:poly_tf_order}, followed by a brief recap of the radius selection procedure. As part of this discussion, we present the closed-form expression for the Fourier coefficients of $\psi$ in \cref{prop:psi_hat}.
%\subsection{Radius}
%\label{sec:poly_tf_radius}

For this reason, we focus on the piecewise polynomial test function family throughout the remainder of the paper. We begin by presenting the closed-form expression for the Fourier coefficients of the reference test function $\psi$ in \cref{prop:psi_hat}, which is used to construct the diagonal matrix $\Psihatbf(r)$ in \cref{alg:SL}. We then turn to an analysis of the polynomial order in \cref{sec:poly_tf_order}, offering general guidelines for its selection. The section concludes with a demonstration of the effectiveness of our hyperparameter selection procedure.

\begin{proposition}
\label{prop:psi_hat}
Let $\psi(t;r, p)$ be defined as in \eqref{eq:tf_poly_main}. The Fourier coefficients $\widehat{\psi}_n$, for $n \in \Zbb$, are real-valued and symmetric, i.e., $\widehat{\psi}_n = \widehat{\psi}_{-n}$, and given explicitly by
\begin{equation*}
\widehat{\psi}_n = 
    \begin{cases}
        \frac{2C}{\sqrt{T}} r^{2p+1} \sum_{j=0}^p \binom{p}{j}\frac{(-1)^j}{2j+1} \quad & n = 0,\\
        \frac{C}{\sqrt{T}} \sqrt{\pi} \big(\frac{rT}{n\pi} \big)^{p+\frac{1}{2}} \Gamma(p+1) J_{p + \frac{1}{2}}\big( \frac{2\pi nr}{T}\big)& n \geq 1,
    \end{cases}
\end{equation*}
where $\Gamma$ denotes the Gamma function and $J_{\nu}$ is the Bessel function of the first kind.
\end{proposition}
\begin{proof}
    This is a straightforward application of a Fourier Tranform. For the convenience of the reader, we have included a complete proof in \cref{sec:psi_hat}.
\end{proof}

\subsection{Order Selection}
\label{sec:poly_tf_order}
In this subsection, we discuss the test function order $p$ by analyzing the integration error $\ebfint(\phi, \ustar)$, as defined in \eqref{eq:eint_u}. Here, $\phi(t)$ is a translation of the piecewise polynomial reference test function $\psi$ introduced in \eqref{eq:tf_poly_main}. This analysis parallels results established in Lemma 2.3 of \cite{MessengerBortz2021MultiscaleModelSimul}, which we revisit here for completeness and to emphasize its relevance in the present context.

The integration error due to trapezoidal discretization is
\begin{equation*}
    \begin{aligned}
         \ebfint(\phi, \ustar)& = \dt \sum_{m=0}^{M}{}'' \phi(t_m) \dot{\ustar}(t_m) + \phidot(t_m) \ustar(t_m) \\
        & =         
         \dt \sum_{m=0}^{M}{}'' 
         \frac{d}{dt} \brackets{\phi\ustar}(t_m),
    \end{aligned}
\end{equation*}
where the notation $\sum{}''$ indicates that the first and last terms are halved before entering the sum. Applying the Euler–Maclaurin formula \eqref{eq:EulerMaclaurin} to $\int_0^T \frac{d}{dt}\brackets{\phi \ustar}dt$ yields
\begin{equation*}
    \begin{aligned}
         \ebfint(\phi, \ustar) = \sum_{s=1}^{\infty} \dt^{2s} \frac{B_{2s}}{(2s)!} \sum_{l=0}^{2s}\binom{2s}{l}  \evalbar{ \parens{ \phi^{(l)} u^{\star (2k-l)}}}{0}^{T}.
    \end{aligned}
\end{equation*}
Because $\phi_k^{(l)}(0) = \phi_k^{(l)}(T) = 0 $ for $0 \leq l \leq p-1$, the lower-order boundary terms in the Euler–Maclaurin expansion vanish.  As a consequence, when $p$ is odd, the leading order integration error scales like $\mathcal{O}(\dt^{p+1})$, while for even $p$ it scales like $\mathcal{O}(\dt^{p})$. In both cases, increasing the order $p$ leads to improved accuracy, with errors potentially reaching machine precision for sufficiently high $p$.  This leads to a practical guideline: the order $p$ can be selected based on the time step $\dt$ to control the integration error. In all of the problems we consider, choosing $p$ between $16$ and $22$ provides sufficiently accurate results. While determining an optimal value of $p$ is outside the scope of this work, our experiments indicate that performance remains stable across a wide range of $p$.

We conclude this subsection by summarizing the recommended strategy for selecting hyperparameters for the piecewise polynomial test functions. One first selects an order $p$, then applies \cref{alg:SL} to compute the corresponding critical radius $\rchat$ and construct the test function matrices accordingly. The effectiveness of this procedure is illustrated in \cref{fig:e_ps} below.

\begin{figure}[htbp]
  \centering
    \includegraphics[width=0.8\textwidth]{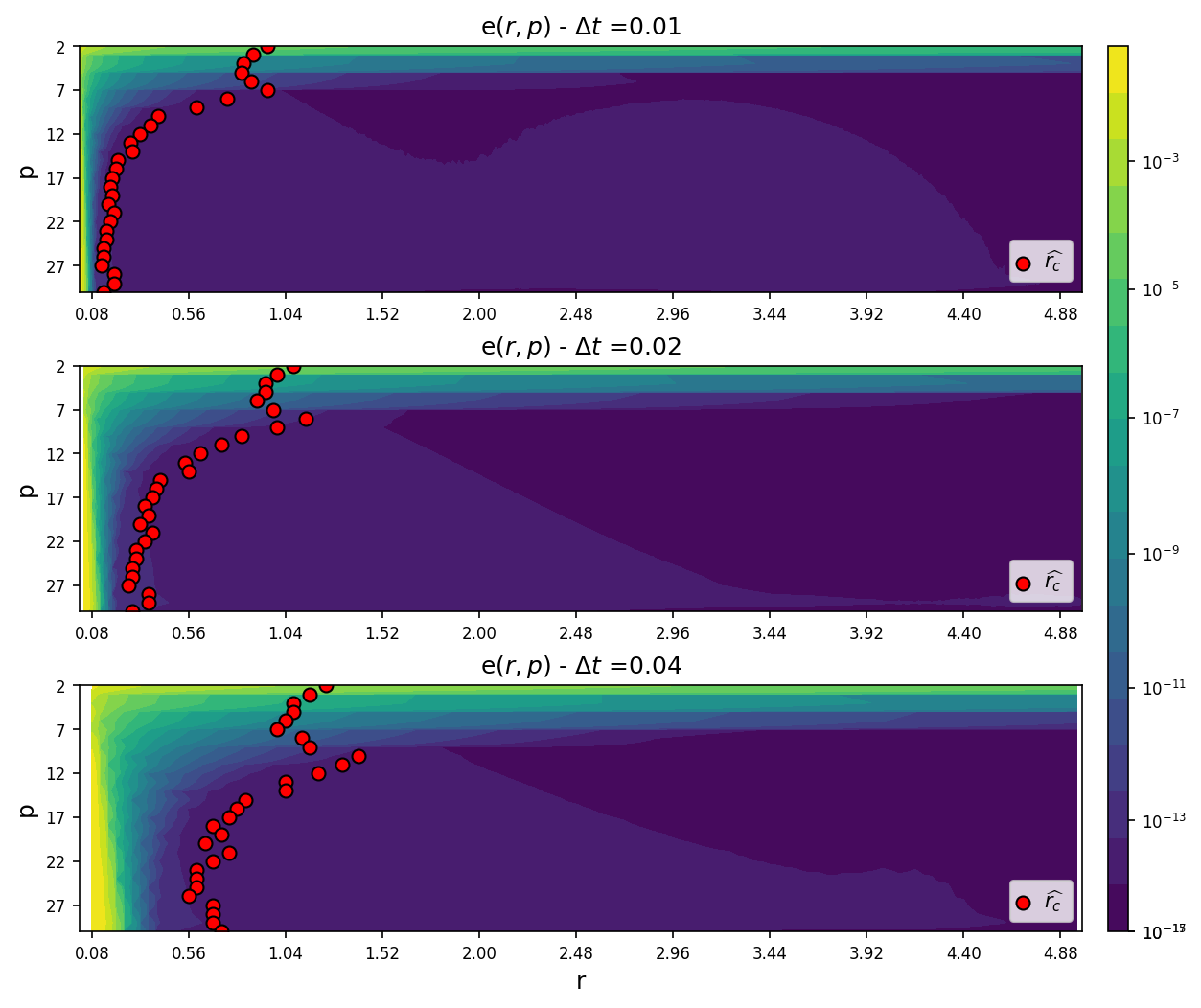}
  \caption{Integration error $\erm(r, p)$  \eqref{eq:e} for the logistic growth model, plotted on a logarithmic scale as a function of test function radius $r$ and order $p$, for $\dt = 0.01$, $0.02$, and $0.04$. The red scatter markers indicate the critical radius $\rchat$ identified by \cref{alg:SL}. Across all resolution levels, the integration error decreases toward machine precision as the order $p$ increases. For sufficiently large $p$, the critical radius $\rchat$ lies within the region of minimal integration error.}
  \label{fig:e_ps}
\end{figure}

In \cref{fig:e_ps}, we plot the integration error $\erm(r, p)$ \eqref{eq:e} as a function of the test function radius $r$ and order $p$, for time step $\dt = 0.01$, $0.02$, and $0.04$, corresponding to $M = 1000, 500$ and $250$, respectively. In all cases, we observe that for fixed values of $r$, the integration error $\ebf$ decreases rapidly with increasing $p$, until reaching machine precision. Additionally, consistent with the behavior described in \cref{sec:radiusselection}, we see that for fixed $p$, increasing $r$ also reduces the integration error. As $\dt$ increases, higher values of $p$ are required for the integration error to decay to machine precision, reflecting the fact that the error scales at least as $\mathcal{O}(\dt^{p})$. We also emphasize that in the presence of measurement noise, the integration error becomes negligible compared to the noise level, as shown in \cref{sec:results}. In such cases, using very high $p$ may have little effect on the final parameter estimates, since the dominant error source is no longer the numerical integration.
Finally, in \cref{fig:e_ps}, we also plot the optimal radius identified using \cref{alg:SL}. For sufficiently large $p$, the critical radius $\rchat$ consistently lies within the region of minimal integration error.

\section{Results}

\label{sec:results}
In this section, we evaluate the effectiveness of our algorithm for selecting test function hyperparameters by analyzing how the WENDy framework performs as a function of the test function radius and polynomial (see \cref{sec:result_Sl}). In addition, in \cref{sec:compareMG},  we compare two strategies for constructing the test function set $\Kcalbf$: the Single-scale-Local (SL) approach described in \cref{alg:SL}, and  Multi-scale-Global (MG) approach from \cref{alg:MG}.

Both objectives are explored in two regression settings: Ordinary Least Squares (WENDy-OLS) (see \cref{sec:weakform}) and Iteratively Reweighted Least Squares (WENDy - IRLS) (see \cref{sec:WENDy}).  Results are demonstrated on four canonical ODE systems: logistic growth, Duffing oscillator, Lorenz system, and FitzHugh–Nagumo. The governing equations and parameter values used in these experiments are summarized in \cref{table:eq}.

To assess performance, we perform $100$ independent trials for each setting. In each trial, Gaussian noise with standard deviation $\sigma = \sigma_{NR}\norm{ \Ubfstar}_F$ is added to the data, where $\sigma_{NR}$ denotes the noise ratio. We evaluate two error metrics: the relative coefficient error of the IRLS estimate $E_2(\widehat{\wbf})$ and the relative coefficient error of the OLS estimate $E_2(\wbf_0)$.
%, and the forward simulation error $ E_{FS}(\widehat{\wbf})$, computed by simulating the learned model from the true initial condition $\ubf^{\star}(0)$. Unless otherwise noted, all forward simulations are performed using the RK45 scheme with absolute and relative tolerances set to $10^{-15}$.
These error metrics are defined as follows:
\begin{equation}
\begin{aligned}
    E_2(\widehat{\wbf}) := \frac{\norm{\widehat{\wbf} - \wbf^{\star}}_2}{\norm{\wbf^{\star}}_2}, \quad   E_2(\wbf_0) := \frac{\norm{\wbf_0 - \wbf^{\star}}_2}{\norm{\wbf^{\star}}_2}, 
    %\quad E_{FS}(\widehat{\wbf}) := \frac{\norm{\Ubfstar - \widehat{\Ubf}}_F}{\norm{\Ubfstar}_F},
\end{aligned}
\label{eq:metric}
\end{equation}
where $\wstar$ denotes the true parameter vector: $\wstar:= \textsf{vec}(\Wbfstar)$.

\begin{table}[h!]
\centering
 \begin{tabular}{l@{\hspace{0.5em}}c@{\hspace{0.5em}}c}
\toprule
\addlinespace
 \textbf{Equation} & \textbf{ODE} & \textbf{Parameters}\\ 
 \addlinespace
 \midrule
 \addlinespace
  \begin{tabular}{@{}c@{}} \textbf{Logistic} \\ \textbf{Growth} \end{tabular} & $\dot{u} = w_1u + w_2u^2$ & \begin{tabular}{@{}c@{}} $T = 10$ \\ $u(0) = 0.01$ \\ $\wbf^{\star} = [1, -1]^T$   \end{tabular} \\
\addlinespace
 \midrule
 \addlinespace
 \textbf{Duffing} & $\begin{cases} \dot{u}_1 &= w_1u_2 \\ \dot{u}_2 &= w_2u_2 + w_3u_1 + w_4u_1^3 \end{cases}$ & \begin{tabular}{@{}c@{}} $T = 20$ \\ $u(0) = [0, 2] $ \\ $\wbf^{\star} = [1, -0.2, -0.05, -1]^T$  \end{tabular} \\
 \addlinespace
 \midrule
 \addlinespace
  % \textbf{Van der Pol}  & $\begin{cases} \dot{u}_1 &= w_1u_2 \\ \dot{u}_2 &= w_2u_2 + w_3u_1^2u_2 + w_4u_1 \end{cases}$  & \begin{tabular}{@{}c@{}} $T = 20$ \\ $u(0) = [0, 1] $ \\ $\wbf^{\star} = [1, 2, -2, -1]^T$  \end{tabular} \\
%\addlinespace
 %\midrule
 \addlinespace
 \begin{tabular}{@{}c@{}} \textbf{FitzHugh} \\ \textbf{- Nagumo} \end{tabular} & $\begin{cases} \dot{u}_1 &= w_1u_1 + w_2u_1^3 + w_3u_2 \\ 
 \dot{u}_2 &= w_4u_1 + w_5 + w_6u_2 \end{cases}$ 
 &  \begin{tabular}{@{}c@{}} $T = 25$ \\ $u(0) = [0, 0.1] $ \\ $\wbf^{\star} = [3, -3, 3, -\frac{1}{3}, \frac{17}{150}, \frac{1}{15}]^T$  \end{tabular}\\
 \addlinespace
 \midrule
 \addlinespace
  \textbf{Lorenz} & $\begin{cases} \dot{u}_1 &= w_1u_2 + w_2u_1 \\ 
 \dot{u}_2 &= w_3u_1 + w_4u_1u_3 + w_5u_2 \\ 
 \dot{u}_3 &= w_6u_1u_2 + w_7 
 \end{cases}$ &  \begin{tabular}{@{}c@{}} $T = 10$ \\ $u(0) = [-8, 10, 27] $ \\ $\wbf^{\star} = [10, -10, 28, -1, -1, 1, -\frac{8}{3}]^T$  \end{tabular}\\
 \end{tabular}  
 \caption{Equations used in numerical experiments.}
 \label{table:eq}
\end{table}

\subsection{Effectiveness of the Single-scale-Local (SL) Approach}
\label{sec:result_Sl}
We first evaluate the effectiveness of the proposed algorithm for selecting the critical radius $\rchat$ by analyzing the performance of the WENDy framework across a range of radius values and noise levels. \Cref{fig:logistic_growth_rad,fig:duffing_rad,fig:lorenz_rad,fig:fitzhugh_nagumo_rad} show the median parameter error under WENDy-OLS (left), i.e., $E_2(\wbf_0)$, and WENDy-IRLS (right), i.e., $E_2(\what)$, as defined in \eqref{eq:metric}.  Results are shown for four canonical systems: the logistic growth equation, Duffing oscillator, Lorenz system, and FitzHugh–Nagumo model. Each plot displays results across nine different noise levels. Each plot displays performance across nine noise levels, with fixed resolution $M = 500$ and polynomial order $p = 16$. In each case, the dashed vertical line indicates the average critical radius $\rchat$ computed using \cref{alg:SL}.

\begin{figure}[htbp]
\centering
    \includegraphics[width=0.8\textwidth]{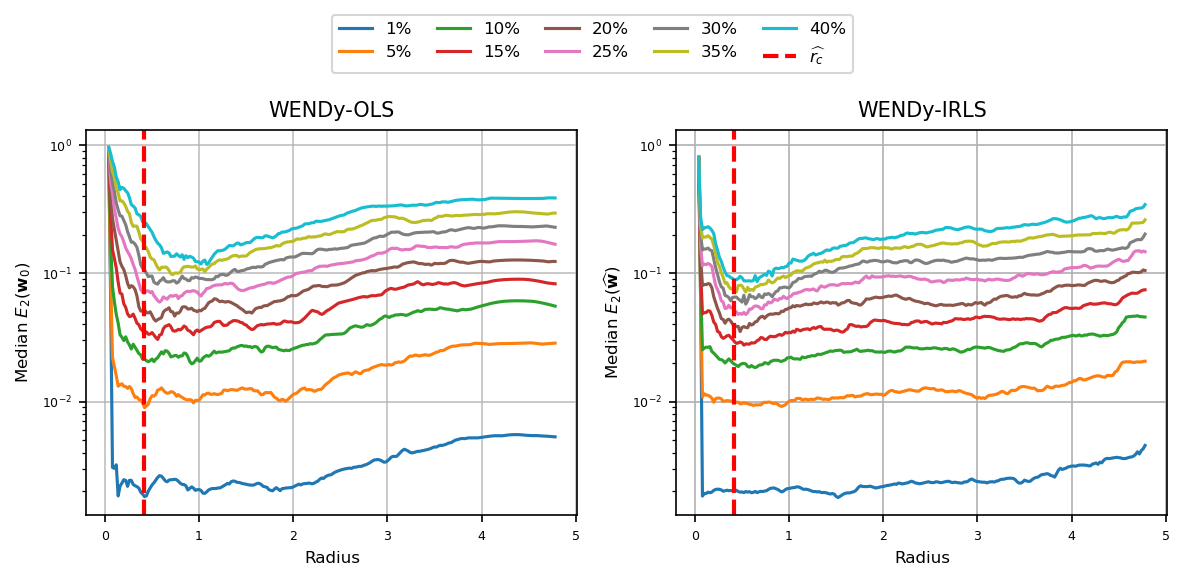}
  \caption{Logistic Growth: Parameter error under WENDy-OLS (left) and WENDy-IRLS (right) as a function of test function radius across varying noise levels. Median relative error is plotted against test function radius for $9$ noise levels, using polynomial order $p = 16$ and temporal resolution $M=500$. The dashed vertical line indicates the mean critical radius $\rchat$ obtained from \cref{alg:SL}. 
  %Under WENDy-OLS, $\rchat$ aligns with the optimal radius at low noise. Under WENDy-IRLS, $\rchat$  consistently identifies the minimal-error region across all noise levels.
  }
\label{fig:logistic_growth_rad}
\end{figure}
\begin{figure}[htbp]
\centering
    \includegraphics[width=0.8\textwidth]{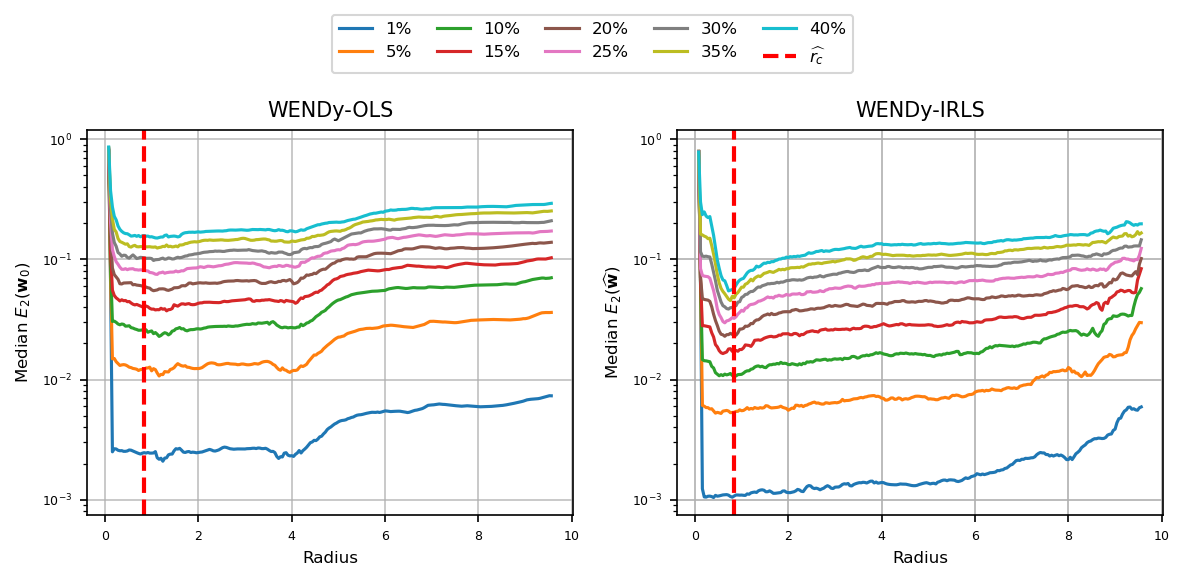}
  \caption{Duffing Equation: Parameter error under WENDy-OLS and WENDy-IRLS as a function of test function radius across varying noise levels. See \cref{fig:logistic_growth_rad} for plot details.}
\label{fig:duffing_rad}
\end{figure}
\begin{figure}[htbp]
\centering
    \includegraphics[width=0.8\textwidth]{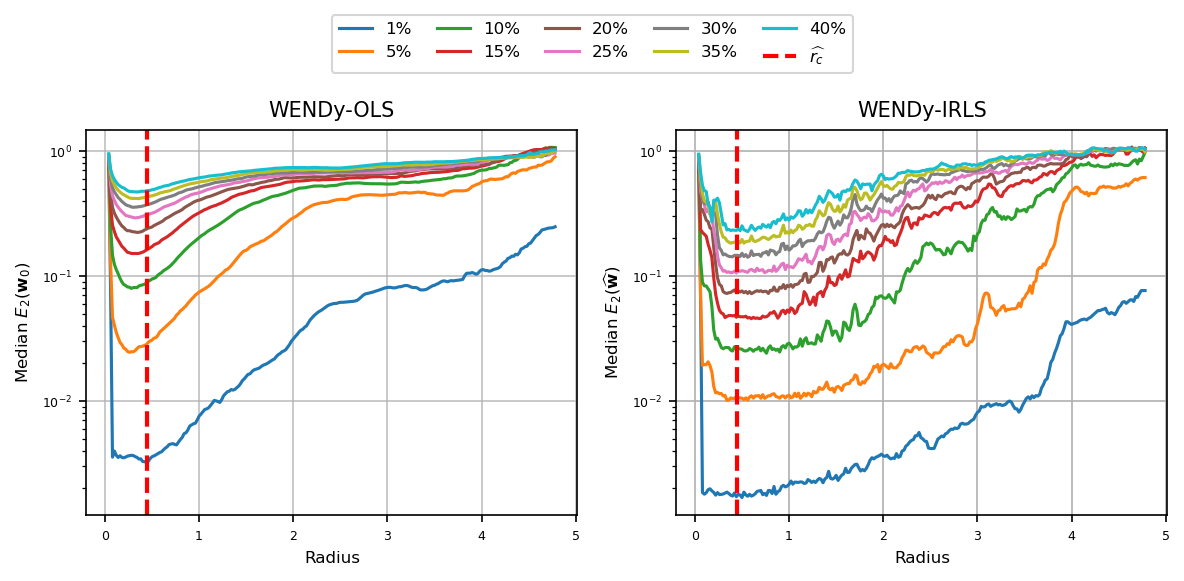}
  \caption{Lorenz Equation: Parameter error under WENDy-OLS and WENDy-IRLS as a function of test function radius across varying noise levels. See \cref{fig:logistic_growth_rad} for plot details.}
\label{fig:lorenz_rad}
\end{figure}
\begin{figure}[htbp]
\centering
    \includegraphics[width=0.8\textwidth]{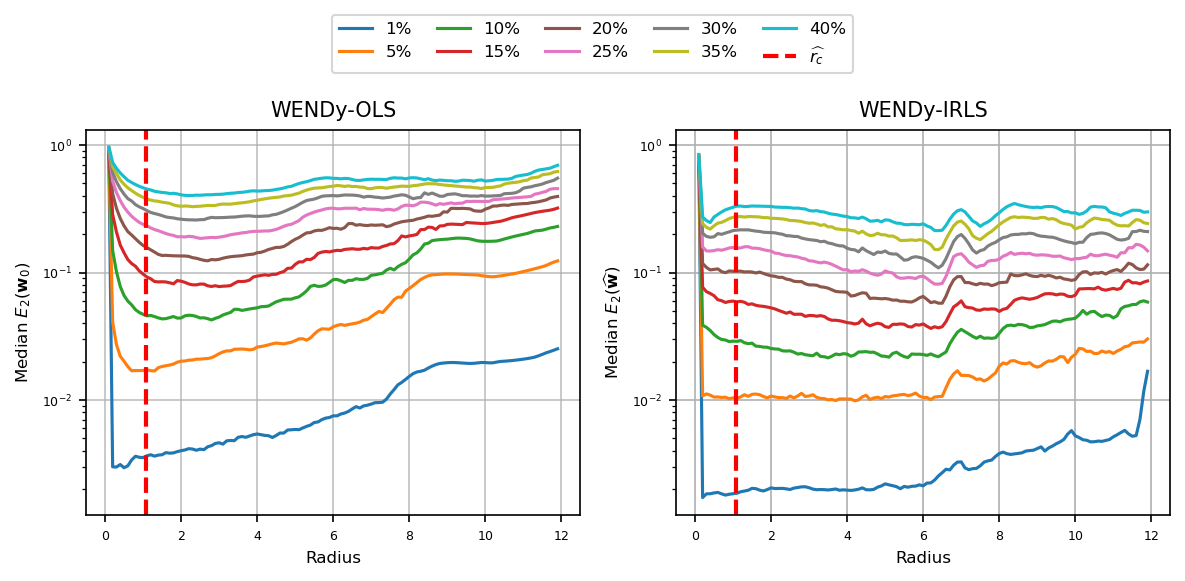}
  \caption{Fitzhugh-Nagumo Equation: Parameter error under WENDy-OLS and WENDy-IRLS as a function of test function radius across varying noise levels. See \cref{fig:logistic_growth_rad} for plot details.}
\label{fig:fitzhugh_nagumo_rad}
\end{figure}

In the OLS setting (left plots), the critical radius $\rchat$ aligns well with the region of minimal parameter error for both the Duffing equation (\cref{fig:duffing_rad}) and the Lorenz system (\cref{fig:lorenz_rad}) across all noise levels. For the logistic growth and FitzHugh-Nagumo systems (\cref{fig:logistic_growth_rad,fig:fitzhugh_nagumo_rad}), this only holds at lower noise levels (e.g., $1\%$ to $5\%$). As noise increases, the location of the minimum-error region shifts, due to the increased influence of nonlinear terms: $u^2$ in the logistic growth model and $u_1^3$ in the FitzHugh–Nagumo model.

In the IRLS setting (right plots), the critical radius $\rchat$ continues to identify the low-error region for both Duffing and Lorenz (\cref{fig:duffing_rad,fig:lorenz_rad}). Additionally, the bias introduced by the nonlinearity $u^2$ in logistic growth is successfully removed via the first-order covariance expansion in WENDy-IRLS, leading to strong alignment between $\rchat$ and the minimum-error region across all noise levels (\cref{fig:logistic_growth_rad}). However, the IRLS correction does not fully address the cubic nonlinearity $u_1^3$ in FitzHugh–Nagumo, resulting in reduced accuracy at higher noise levels (\cref{fig:fitzhugh_nagumo_rad}). 

We further investigate the performance of the SL approach across varying test function radii, polynomial orders, and temporal resolutions. In \cref{fig:logistic_growth_rad_p,fig:duffing_rad_p,fig:lorenz_rad_p}, we present the median parameter error under both WENDy-OLS (left columns) and WENDy-IRLS (right columns) for the logistic growth, Duffing, and Lorenz systems, respectively%\footnote{For additional demonstrations across more noise levels, see \cref{fig:LG_1000,fig:LG_500,fig:LG_250} for logistic growth, \cref{fig:Duff_1000,fig:Duff_500} for the Duffing equation, and \cref{fig:Lorenz_500} for the Lorenz system.}.
In all figures, the top row corresponds to $M = 1000$ time points, while the bottom row shows results for $M = 500$. For the logistic growth and Duffing systems, a noise ratio of $20\%$ is used, while the Lorenz system is tested with $25\%$ noise. These plots also include red markers for the median critical radius $\rchat$ from \cref{alg:SL} and white markers for the median minimum radius $\rmin$ obtained via \cref{alg:MG}.

\begin{figure}[htbp]
\centering
    \includegraphics[width=\textwidth]{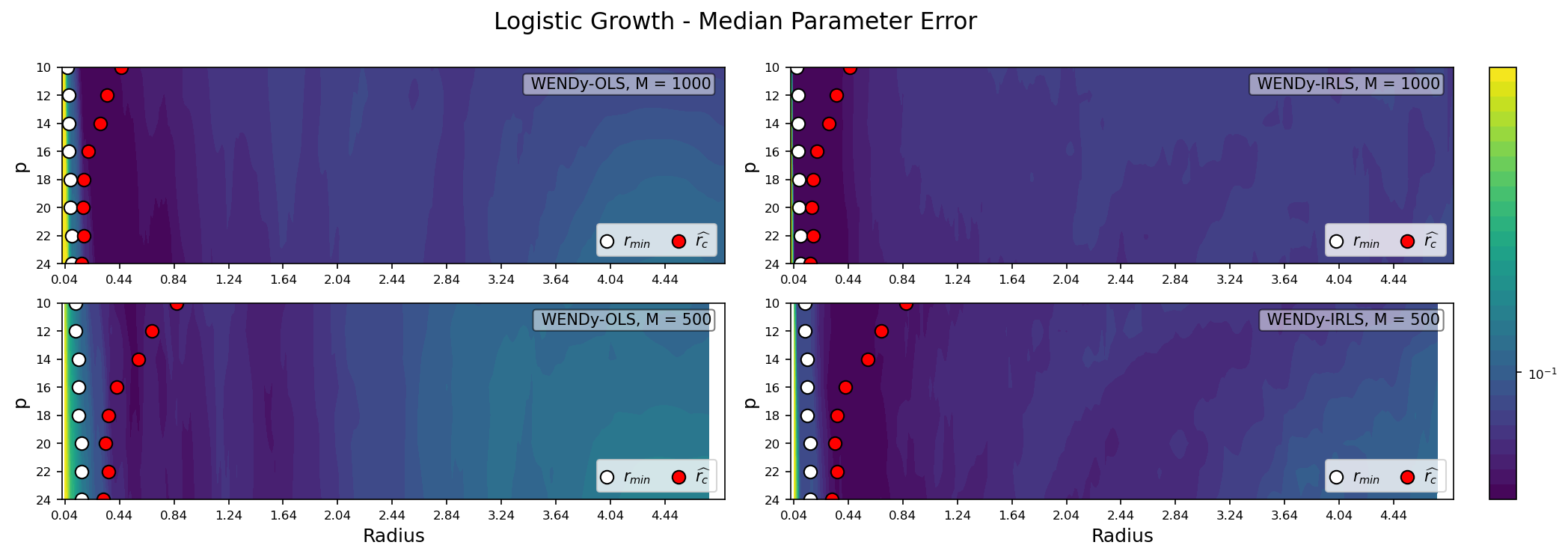}
  \caption{Logistic Growth: Parameter error across radius, order, and temporal resolutions under WENDy-OLS (left) and WENDy-IRLS (right). Median parameter error is shown as a function of test function radius $r$ and order $p$, evaluated at $2$ temporal resolutions: $M=1000$ (top row); $M=500$ (bottom row), with $20\%$ noise. Red markers indicate the median critical radius $\rchat$ computed using \cref{alg:SL}, while white markers denote the minimum radius $\rmin$ obtained via \cref{alg:MG}. 
  %For IRLS,  $\rchat$ consistently aligns with regions of lowest error across all values of $p$ and temporal resolutions.
  }
\label{fig:logistic_growth_rad_p}
\end{figure}
\begin{figure}[htbp]
\centering
    \includegraphics[width=\textwidth]{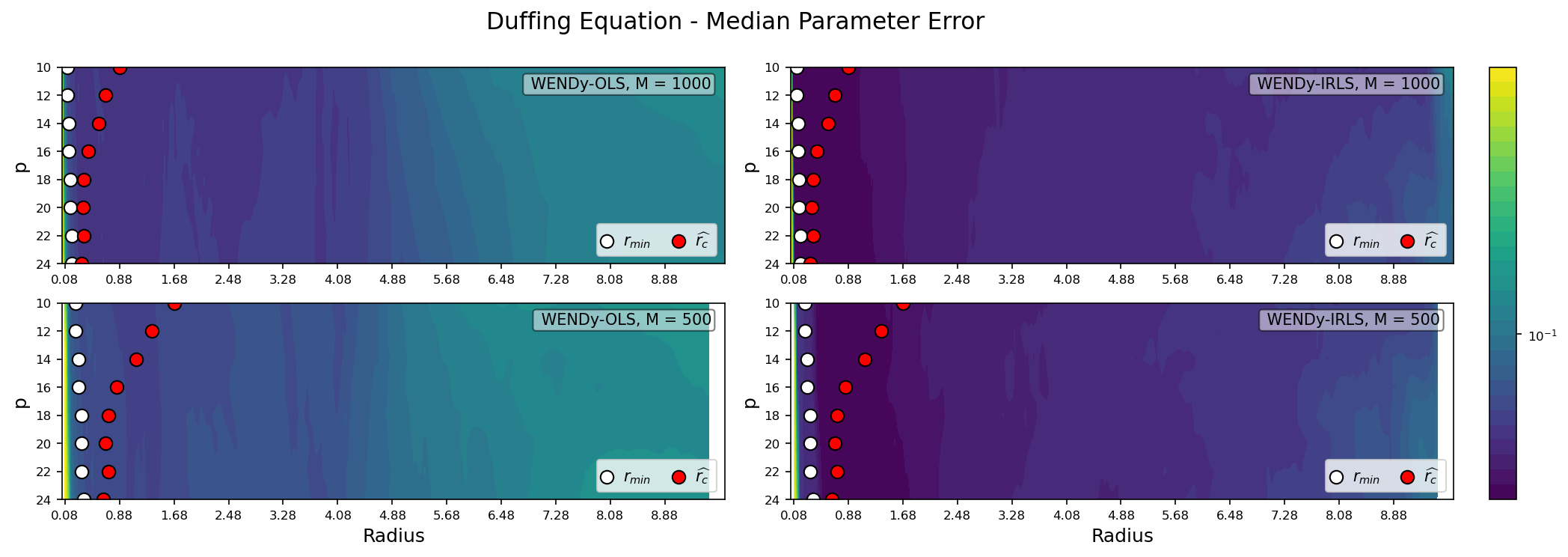}
  \caption{Duffing Equation: Parameter error across radius, order, and temporal resolutions under WENDy-OLS and WENDy-IRLS. See \cref{fig:logistic_growth_rad_p} for plot details.}
\label{fig:duffing_rad_p}
\end{figure}
\begin{figure}[htbp]
\centering
    \includegraphics[width=\textwidth]{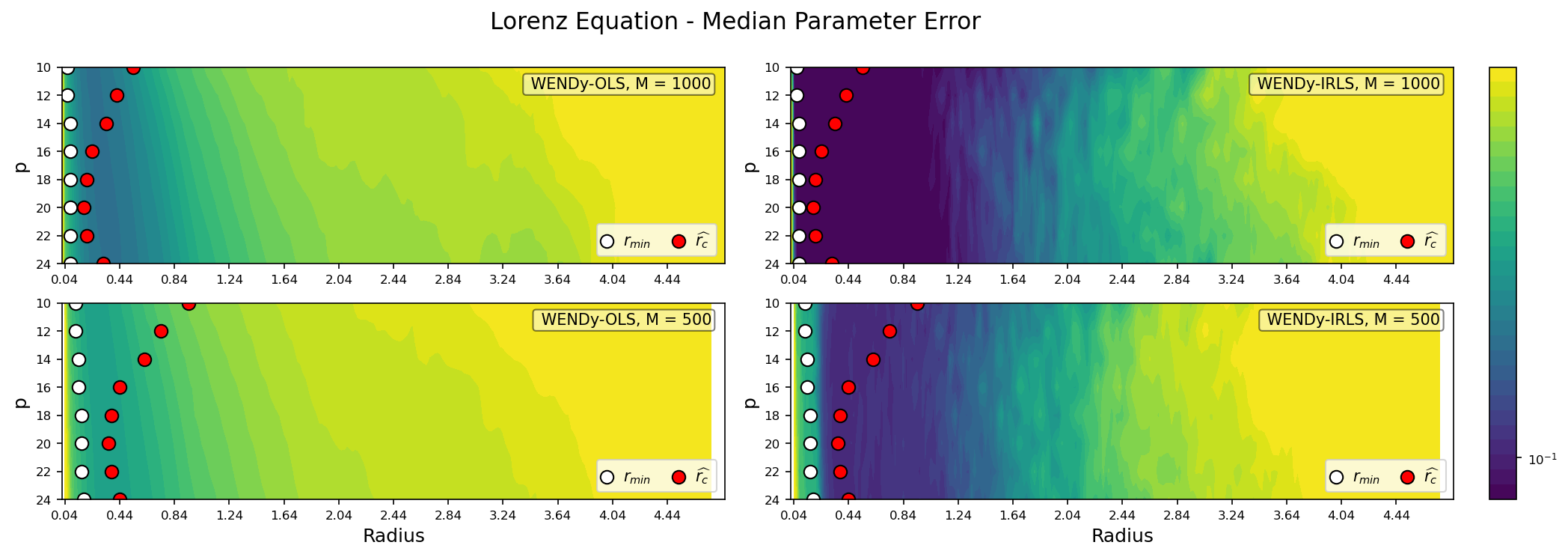}
  \caption{Lorenz Equation: Parameter error across radius, order, and temporal resolutions under WENDy-OLS and WENDy-IRLS. See \cref{fig:logistic_growth_rad_p} for plot details.}
\label{fig:lorenz_rad_p}
\end{figure}

As expected, for the Duffing and Lorenz systems (\cref{fig:duffing_rad_p,fig:lorenz_rad_p}), the red markers representing the critical radius $\rchat$ computed via \cref{alg:SL} consistently fall within regions of low parameter error across all polynomial orders $p$, under both regression schemes (OLS and IRLS), and for both resolution levels.
%This indicates that the selected radius $\rchat$ reliably captures the low-error region in these cases.
For the logistic growth system (\cref{fig:logistic_growth_rad_p}), $\rchat$ deviates from the low-error region in the OLS setting due to bias from the quadratic nonlinearity $u^2$. However, in the IRLS setting, where the covariance is successfully corrected, $\rchat$ aligns well with the region of minimal error. This demonstrates that the Single-scale-Local scheme performs effectively in identifying a reliable radius, particularly when paired with a covariance-correct regression scheme.

\subsection{Comparison with Multi-scale-Global (MG)}
\label{sec:compareMG}
In this subsection, we compare two strategies for constructing the set of test functions: the Multi-scale-Global (MG, \cref{alg:MG}) and Single-scale-Local (SL, \cref{alg:SL}) approaches. The comparison is based on the performance of the WENDy framework in terms of both parameter error and runtime. All computations were performed in parallel on the University of Colorado Boulder's Blanca Condo Cluster, with each trial executed on a single core of an Intel Xeon Platinum 8592+ CPU @ 2.90GHz.

\Cref{fig:logistic_growth_MGSL,fig:duffing_MGSL,fig:lorenz_MGSL,fig:fitzhugh_nagumo_MGSL} present scatter plots of parameter error versus runtime for WENDy-OLS and WENDy-IRLS applied to the logistic growth, Duffing, Lorenz, and FitzHugh–Nagumo systems, respectively. Each plot includes results for both the MG and SL test function set construction strategies with $p=16$, evaluated across two temporal resolutions ($M = 1000$ and $M = 500$) and three noise levels. Specifically, we test noise ratios of $10\%, 20\%$, and $40\%$ for the logistic growth and Duffing equations, and $10\%, 20\%$, and $30\%$ for the Lorenz and FitzHugh–Nagumo systems. Each configuration is run with $100$ independent WENDy realizations. The resulting scatter plots display four sets of points per plot (OLS-MG, OLS-SL, IRLS-MG, IRLS-SL), and white-filled markers indicate the median values for each method.

\begin{figure}[htbp]
\centering
    \includegraphics[width=\textwidth]{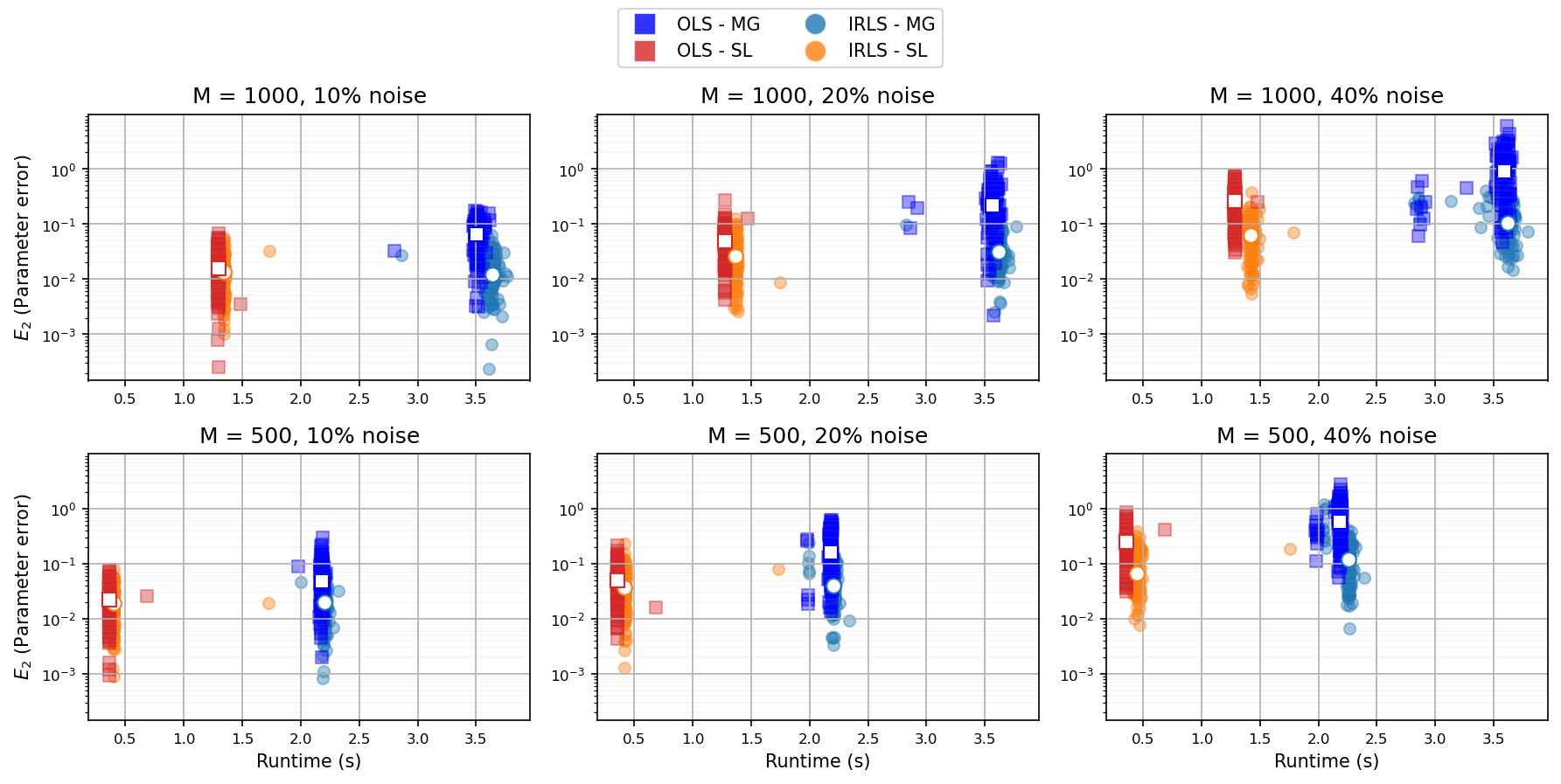}
  \caption{Logistic Growth: Comparison of parameter error and walltime between Multi-scale-Global (MG, \cref{alg:MG}) and Single-scale-Local (SL, \cref{alg:SL}) approaches across different noise levels and temporal resolutions. Top row: $M=1000$; bottom row: $M=500$. Columns correspond to noise levels of $10\%, 20\%$, and $40\%$. Each plot shows the parameter error and runtime of $100$ WENDy runs using OLS and IRLS under MG and SL, totaling four configurations per plot. 
  %Median values are indicated by white-filled markers with matching shapes. Across all settings, the SL approach consistently achieves lower runtime and parameter error than MG.
  }
\label{fig:logistic_growth_MGSL}
\end{figure}
\begin{figure}[htbp]
\centering
    \includegraphics[width=\textwidth]{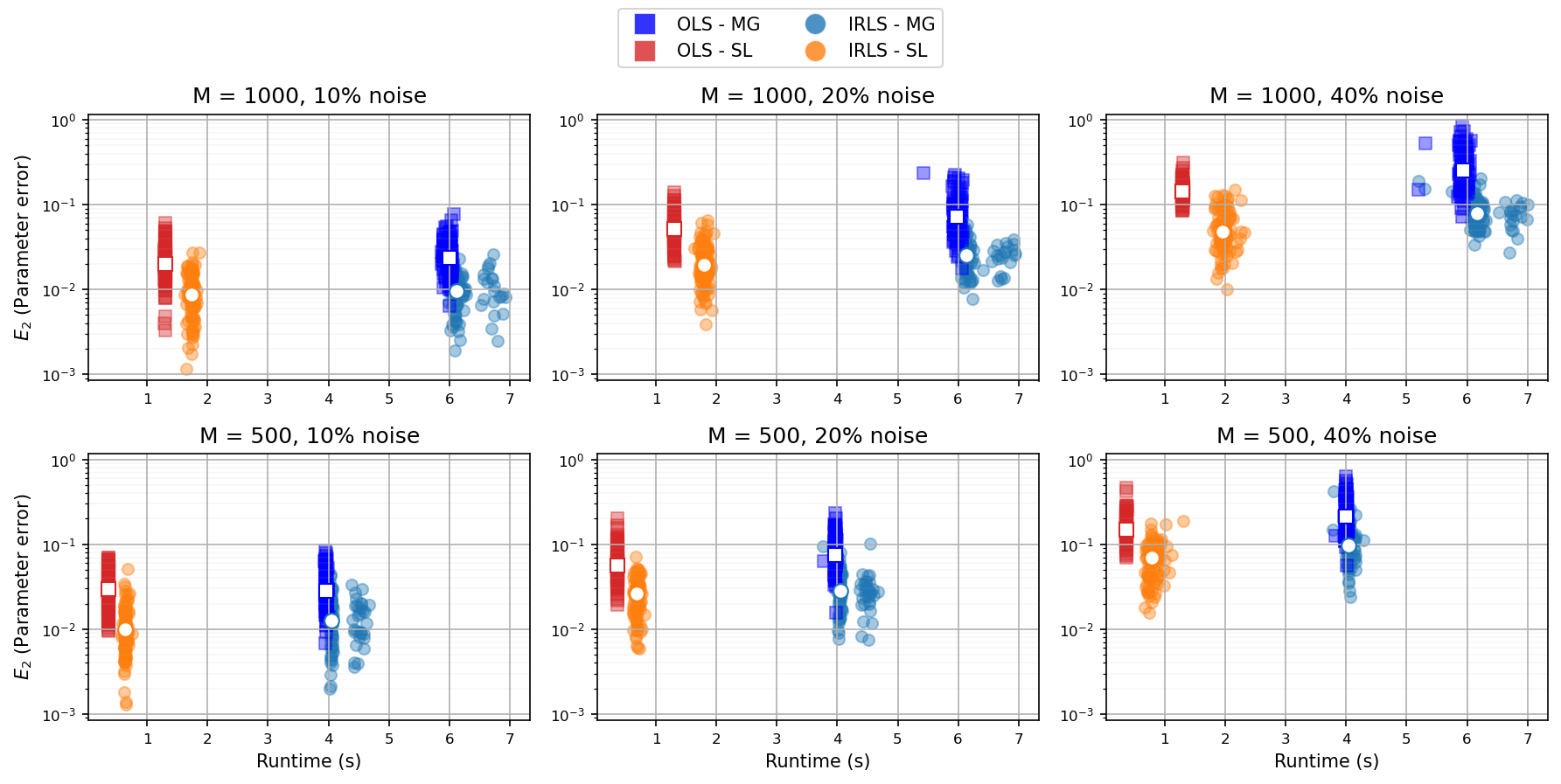}
  \caption{Duffing Equation: Comparison of parameter error and walltime between Multi-scale-Global (MG, \cref{alg:MG}) and Single-scale-Local (SL, \cref{alg:SL}) approaches across different noise levels and temporal resolutions. See \cref{fig:logistic_growth_MGSL} for more details.}
\label{fig:duffing_MGSL}
\end{figure}
\begin{figure}[htbp]
\centering
    \includegraphics[width=\textwidth]{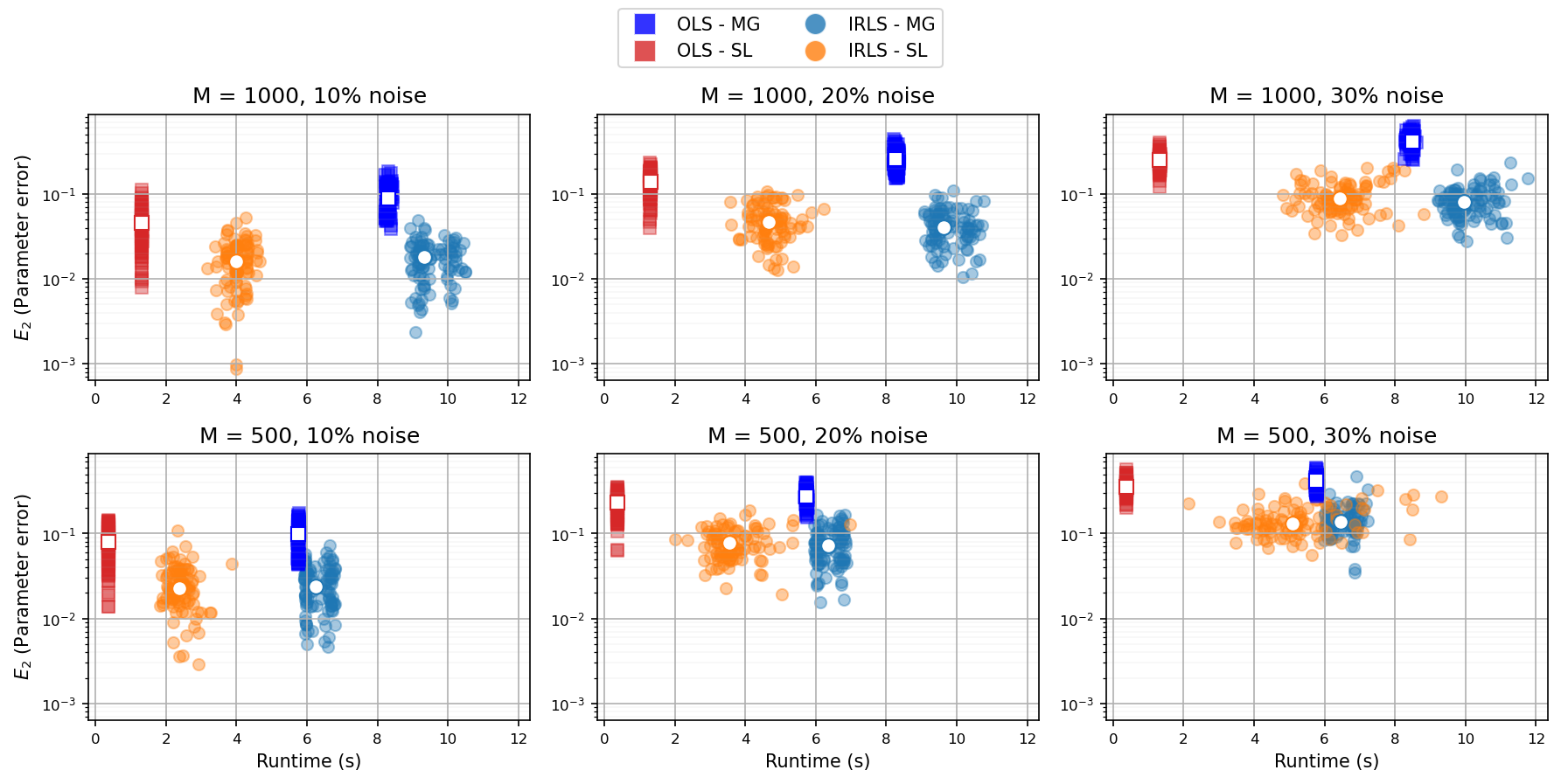}
  \caption{Lorenz Equation: Comparison of parameter error and walltime between Multi-scale-Global (MG, \cref{alg:MG}) and Single-scale-Local (SL, \cref{alg:SL}) approaches across different noise levels and temporal resolutions. See \cref{fig:logistic_growth_MGSL} for more details.}
\label{fig:lorenz_MGSL}
\end{figure}
\begin{figure}[htbp]
\centering
    \includegraphics[width=\textwidth]{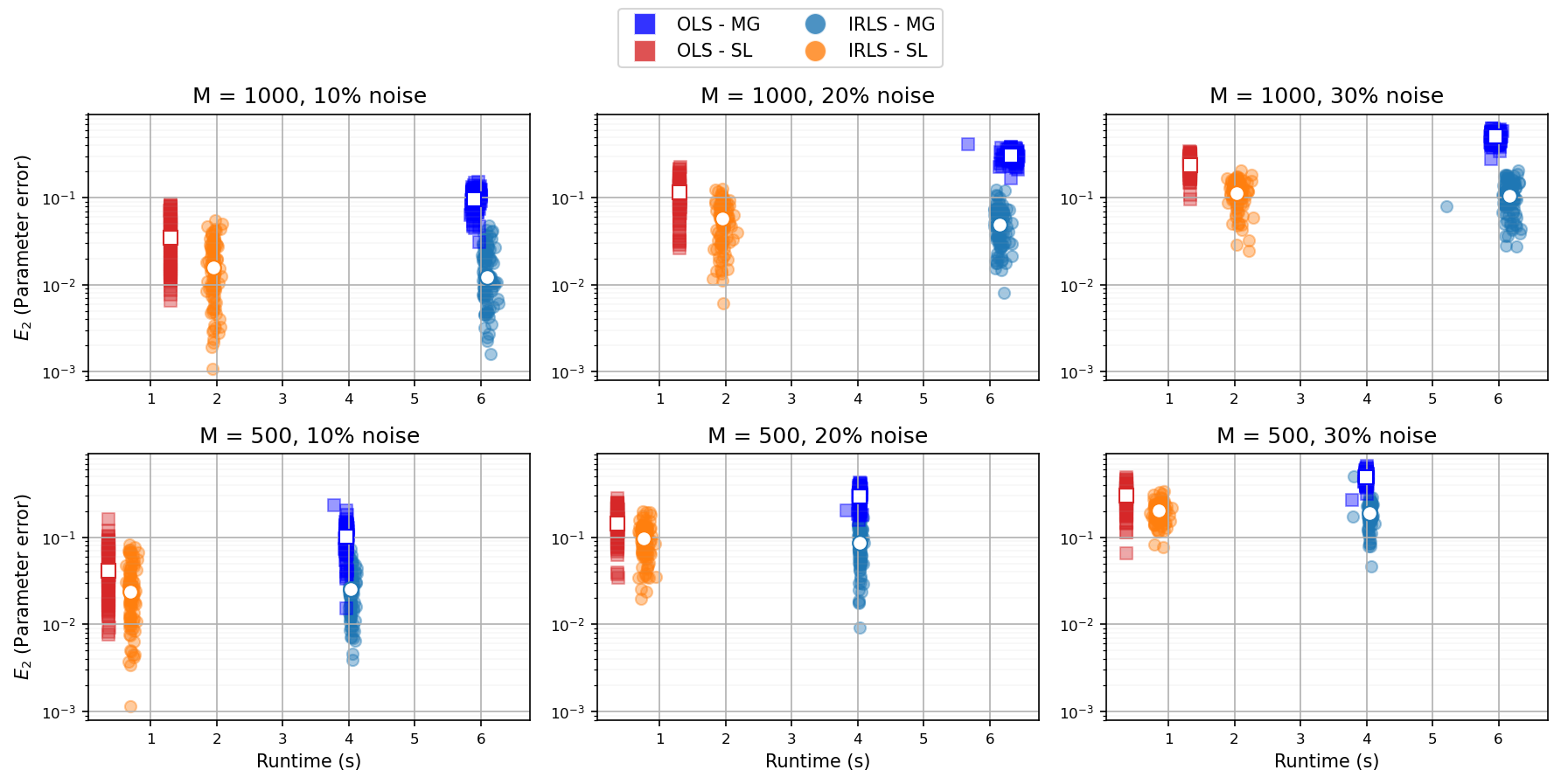}
  \caption{FitzHugh-Nagumo: Comparison of parameter error and walltime between Multi-scale-Global (MG, \cref{alg:MG}) and Single-scale-Local (SL, \cref{alg:SL}) approaches across different noise levels and temporal resolutions. See \cref{fig:logistic_growth_MGSL} for more details.}
\label{fig:fitzhugh_nagumo_MGSL}
\end{figure}

Across all four systems and both regression methods, the Single-scale-Local (SL, \cref{alg:SL}) approach is consistently faster than the Multi-scale-Global (MG, \cref{alg:MG}). The difference is especially pronounced for the logistic growth, Duffing, and FitzHugh–Nagumo systems (\cref{fig:logistic_growth_MGSL,fig:duffing_MGSL,fig:fitzhugh_nagumo_MGSL}) where IRLS-SL achieves roughly $60–70\%$ improvement in median runtime over IRLS-MG. The improved efficiency of Single-scale-Local can be attributed to two main factors. First, selecting the critical radius $\rchat$ is inherently faster than determining $\rmin$, even though both rely on a grid search. Second, the Single-scale-Local approach avoids the overhead of computing an orthonormal basis for the test function matrix, which requires an SVD step in the Multi-scale-Global method.

However, the MG approach can sometimes offer runtime advantages, as seen in the Lorenz system under WENDy-IRLS(\cref{fig:logistic_growth_MGSL}). In this case, the iterative nature of IRLS leads to more iterations at higher noise levels. The orthonormalized test function matrix used in MG reduces numerical conditioning issues and allows faster convergence, which offsets the initial SVD overhead. To further improve the runtime of the SL scheme in such cases, one could consider strategies such as subsampling the test functions or increasing the spacing (gap) between test function centers. However, these modifications are beyond the scope of this work and are left for future investigation.

In terms of parameter error, we observe that for WENDy-OLS, the Single-scale-Local (SL) approach consistently outperforms the Multi-scale-Global (MG) method across all cases tested. However, under WENDy-IRLS, the SL approach shows clear advantages only for the logistic growth and Duffing systems, while performance is nearly identical to MG for the Lorenz and FitzHugh–Nagumo models. This result is consistent with observations from \cref{fig:logistic_growth_rad_p,fig:duffing_rad_p,fig:lorenz_rad_p}, where the critical radius $\rchat$ and the minimum radius $\rmin$, as well as multiples of $\rmin$, fall within regions of low error. As a result, both SL and MG achieve near-optimal performance in those cases.

It is possible that the use of orthonormal test functions in MG contributes to improved convergence in IRLS, especially under high noise conditions. However, further investigation is needed to fully understand this effect, which we leave for future work.

\section{Conclusions}
\label{sec:conclusions}
Scientific Machine Learning (SciML) is a framework for blending well-established computational science methodologies with modern machine learning techniques. The recently developed weak form version (WSciML) leverages a weighted integral transform to yield highly robust algorithms. In particular, the Weak form Sparse Identification of Nonlinear Dynamics (WSINDy) for equation learning and Weak form Estimation of Nonlinear dynamics (WENDy) for parameter estimation both exhibit strong robustness. However, how to choose the test function (i.e., the weight function) in a data-driven manner has been an open question for some time. This paper is the first dedicated solely to answering this question.

In this work, we proposed a strategy for selecting the radius of Single-scale-Local (SL) test function set within the WENDy framework, with the goal of optimizing parameter recovery. Central to our approach is \cref{conj:optimal_radius}, which states that the optimal radius for minimizing expected parameter error aligns with the changepoint in the weak-form integration error. Motivated by this conjecture, we perform a grid search over candidate radii, guided by a surrogate approximation of the weak-form integration error,
$$\Gbfstar\wstar - \bbf^{\star} = \Phibf\Thetabf(\ubf^\star)\wstar + \Phidotbf \ubf^\star, $$
without having access to either the true parameter $\wstar$ or noise-free data $\ubf^\star$. This approximation arises from an asymptotic expansion of the error introduced by discretizing the weak-form integrals. Specifically, for the piecewise polynomial test function, we provide detailed guidance on selecting both the radius $r$ and the order $p$. Our method consistently identifies regions of low parameter error across a wide range of systems, time resolutions, and regression schemes. Compared to the previous Multi-scale-Global (MG) approach, our method achieves significantly faster runtimes while maintaining comparable or superior accuracy, especially under high-noise regimes.

While \cref{conj:optimal_radius} is supported empirically, a rigorous mathematical or statistical justification remains open. Proving this conjecture, perhaps in the limit of continuum data, would substantially strengthen the theoretical foundations of our method.

Although our focus has been on radius selection, several important directions remain unexplored. We did not consider optimizing other aspects, such as the overlap between neighboring test functions or the placement of their centers. A rigorous treatment of these factors presents an interesting direction for future work. In addition, our proposed approximated integration error depends on several technical parameters, such as the truncation order in the Euler–Maclaurin expansion, the number of retained Fourier modes, and the finite difference order. While these parameters were selected empirically in this work, a better understanding of their influence, along with strategies for joint optimization, would provide more robust guidelines for our method.

\section*{Acknowledgments}
This research was supported by a US DOE Mathematical Multifaceted Integrated Capability Center (MMICC) grant to the University of Colorado (DE-SC0023346). This work also
utilized the Blanca condo computing resource at the University of Colorado Boulder. Blanca is jointly funded by computing users and the University of Colorado Boulder.

The authors would also like to thank D.~Messenger (Los Alamos National Lab) for comments on an earlier draft of this manuscript and V.~Dukic (University of Colorado) for insightful discussions concerning statistical concepts.

\bibliographystyle{siamplain}
\bibliography{Tran02TestFunctions}
\appendix
\section{Weak form Parameter Estimation with Iteratively Re-weighted Least Squares}
\label{sec:WENDy}
Here we provide the details of the WENDy-IRLS algorithm as described in \cite{BortzMessengerDukic2023BullMathBiol}.

To begin, we redefine the regression matrix $\Gbf$ and define the response vector $\bbf$ to incorporate all $d$ dimensions of the observed data $\Ubf$
\begin{equation*}
    \begin{aligned}       \Gbf := [\mathbb{I}_{d}\otimes\Phibf\Thetabf] \in \Rbb^{Kd \times Jd}, \quad\bbf :=-\mathsf{vec}(\Phidotbf\Ubf) \in \Rbb^{Kd},\quad\wbf := \mathsf{vec}(\Wbf) \in \Rbb^{Jd},
    \end{aligned}
\end{equation*}
where $\otimes$ denotes the Kronecker product. The residual for the one-dimensional case is defined in \eqref{eq:e_int}; we restate it here in the full  $d$-dimensional setting:
\begin{equation*}
    \ebf_{\text{int}}(\ubf^{\star}) := \Gbfstar \wbf^{\star} - \bbf^{\star} = \Phibf \Thetabfstar \wbf^{\star} + \Phidotbf \ubf^{\star}  \in \Rbb^{Kd}.
\end{equation*}
A Taylor expansion of  $\ebf_\Theta-\bbf^{\pmb{\varepsilon}}$ about the data $\Ubf$ yields
\begin{equation*}
\ebf_\Theta-\bbf^{\pmb{\varepsilon}} = \Lbf_{\wbf}\mathsf{vec}(\pmb{\varepsilon})+\hbf(\Ubf,\wbf,\pmb{\varepsilon}),
%\label{eqn:eThetabeps}
\end{equation*}
where $\Lbf_{\wbf}$ is the leading-order linear operator and $\hbf$ contains higher-order terms in $\pmb{\varepsilon}$. Specifically,
\begin{equation*}
    \mathbf{L}_{\wbf} :=[\mathsf{mat}(\wbf)^{T}\otimes\Phi]\nabla\Theta\Kbf+[\mathbb{I}_{d}\otimes\dot{\Phi}],
\end{equation*}
with $\Kbf$ the commutation matrix satisfying $\Kbf \mathsf{vec}(\pmb{\varepsilon}) = \mathsf{vec}(\pmb{\varepsilon}^T)$, and $\nabla \Theta$ the Jacobian of the feature map $\Theta$ with respect to $\Ubf$.
%\end{comment}

Assuming Gaussian noise and negligible integration error, the residual at the true weights $\wbf^\star$ follows the distribution 
\begin{equation*} 
\Gbf \wbf^\star - \bbf \sim \mathcal{N}\left(\mathbf{0}, \sigma^2 \Lbf_{\wbf^\star} \Lbf_{\wbf^\star}^T\right),
%\label{eqn:ResidDistwstar} 
\end{equation*}
This motivates an iterative algorithm in which, at each step $n$ the covariance $\Cbf^{(n)}$  is re-estimated using the current parameter iterate $\wbf^{(n)}$ for the regularization parameter $\alpha$: 
\begin{equation*}
    \begin{aligned}
        \Lbf^{(n)} &\leftarrow [\textsf{mat}(\wbf^{(n)})^{T}\otimes\Phi]\nabla\Theta(\Ubf)\Kbf+[\mathbb{I}_{d}\otimes\dot{\Phi}],\\
        \Cbf^{(n)} &= (1-\alpha)\Lbf^{(n)}(\Lbf^{(n)})^T + \alpha \Ibf,\\
        \wbf^{(n+1)} &\leftarrow (\Gbf^T(\Cbf^{(n)})^{-1}\Gbf)^{-1}\Gbf^{T}(\Cbf^{(n)})^{-1}\bbf.
    \end{aligned}
\end{equation*}
We initialize the iteration with the OLS estimate, setting $\wbf^{(0)} = \wbf_0: = \mathsf{vec}(\Wbf_0)$. A detailed description of the WENDy-IRLS algorithm, along with its strengths and limitations, can be found in \cite{BortzMessengerDukic2023BullMathBiol}.

\section{Changepoint Detection via Piecewise Linear Approximation}
\label{sec:changepoint}
In this section, we describe the changepoint detection method used for a decreasing function $y(x)$ defined on a uniform grid $\{x_0, x_1, \dots, x_M\}$.
%x_m = m \Delta x$, $m = 0, 1, \dots M$.
The procedure, originally introduced in \cite{MessengerBortz2021JComputPhys}, identifies the point where the function transitions from decay to stagnation by selecting the breakpoint that yields the best two-segment linear approximation to $y$. For completeness, we include the full description and implementation in \cref{alg:changepoint}.

\begin{algorithm}
\caption{Changepoint Detection via Piecewise Linear Fit}
\label{alg:changepoint}
\begin{algorithmic}
\STATE \textbf{Input:} Function $y(x)$ defined over the uniform grid $\{x_m\}_{m=0}^M$, with corresponding values $\{y_m\}_{m=0}^M$
%$x_m = m \Delta x$, $m = 0, 1, \dots M$.
\STATE \textbf{Output:} Changepoint $x_c$ of $y(x)$
\STATE For each grid length $k = 1, 2\dots, M-1$:
%corresponding to  $x_1, x_2, \dots, x_{M-1}$
%length $k = 1, 2\dots, M-1$ corresponding to $x_k = k \Delta x$:  
\begin{itemize}
    %\item Compute line segment $L_1^(\cdot)$ from $(x_0, y_0)$ to $(x_k, y_k)$. 
    %\item Compute line segment $L_2(\cdot)$ from $(x_k, y(x_k))$ to $(x_M, y(x_M))$.
    \item Compute line segment from  $(x_0, y_0)$ to $(x_k, y_k)$ 
    $$L_1^k(x):= \frac{y_k - y_0}{x_k-x_0}(x - x_0) + y_0$$
    \item Compute line segment from $(x_k, y(x_k))$ to $(x_M, y(x_M))$
    $$L_2^k(x):= \frac{y_M - y_k}{x_M-x_k}(x - x_k) + y_k$$
    \item Compute error 
$$\Ebf(k) = \sqrt{\sum_{m=0}^k \parens{\frac{ L_1^k(x_m) - y_m}{y_m}}^2
    + \sum_{m=k}^M \parens{\frac{ L_2(x_m) - y(x_m)}{y(x_m)}}^2}. $$
\end{itemize}
Identify the changepoint  $x_c:=  \arg\min_{k} \Ebf(k)$
\RETURN $x_c$
\end{algorithmic}
\end{algorithm}

In this work, \cref{alg:changepoint} is used to compute the changepoint $r_c $ of $\log(\norm{\ebfint(r)}$, as well as the \emph{critical radius} $\widehat{r}_c$ from $\log(\norm{\ebfhatint(r)}$. While many methods exist for identifying such changepoints in the literature, we adopt this simple and interpretable two-line approximation approach in our work for its practical effectiveness. Developing or selecting an optimal changepoint detection algorithm tailored to weak-form integration error curves is left for future work.

\section{Reference Test Function $\psi(t)$}
\label{sec:reference_tf}
We define our set of test functions $\Kcalbf$
%$\{ \phi_k(t) \}_{k=1}^K$ in \cref{def:tf_set}
by translating a reference test function $\psi(t)$. For the interested reader, in this section, we present several of our previously-used test functions.
%In our prior works, we have used two different classes of reference test functions. In this work, we will focus on using the piecewise polynomial test function, but for the convenience of the reader, we will describe both of them here.
In \cite{MessengerBortz2021MultiscaleModelSimul}, we considered a piecewise polynomial reference test function $\psi(t;r,p)$ of the form
\begin{equation}
  \psi(t;r, p) =
\begin{cases}
     C(r-t)^p(r+t)^p \quad & t \in [-r, r],\\
     \quad  \quad 0 \quad & \text{otherwise},
     \label{eq:tf_poly}
\end{cases}
\end{equation} 
where $C$ is chosen so that $\norm{\psi}_{\infty} = 1$. In the same work, we demonstrated that employing higher-order polynomial test functions can significantly improve accuracy, up to machine precision (see \cref{sec:poly_tf_order}). Motivated by this finding, and to avoid tuning $p$, in \cite{BortzMessengerDukic2023BullMathBiol}, we propose the use of a smooth $C^\infty$ bump function
\begin{equation}
\psi(t;r, \eta) = C\exp\left(-\frac{\eta}{[1-(t/r)^2]_+}\right),
\label{eq:tf_CinftyBump}
\end{equation}
where $C$ ensures $\norm{\psi}_2 = 1$, $\eta$ controls the sharpness of the bump, and $[\cdot]+ := \max(\cdot, 0)$ restricts the support to $[-r, r]$. 

In our earlier efforts, we proposed different parameter selection strategies for each test function family. For the piecewise polynomials in \cite{MessengerBortz2021JComputPhys}, we suggested choosing the radius $r$ and order $p$ by treating these test functions as low-pass filters: examining the spectrum of the observed data $\Ubf$, identifying a critical wavenumber beyond which noise dominates, and selecting parameters so that the Fourier transform of $\psi$ decays rapidly past that threshold. For the bump functions used in  \cite{BortzMessengerDukic2023BullMathBiol}, we found that estimation accuracy is largely insensitive to the shape parameter $\eta$, allowing it to be arbitrarily fixed (e.g., $\eta = 9$). However, the performance when using \eqref{eq:tf_CinftyBump} proved unreliable and this we extended it to create a Multi-scale-Global (MG) orthonormal set, using an integration error analysis and an SVD (and summarized in \cref{sec:MinimumRadiusSelection}).

\section{Proof of \cref{lem:a}}
\label{sec:proof_lemma}
\begin{proof}
Consider the integration error $\ebf_{\text{int}}(\phi, \ustar)$
\begin{equation}
    \begin{aligned}
        \ebf_{\text{int}}(\phi, \ustar) &= \dt \sum_{m=0}^{M}{}'' \phi(t_m)\dot{u}^{\star}(t_m) + \phidot(t_m)\ustar(t_m), 
    \end{aligned}
    \label{eq:eint_u}
\end{equation}
where the double prime notation on the sum $\sum{}''$ denotes the trapezoidal-weighted sum, i.e., the first and last terms are halved before summing. Expressing $\phi(t)$ and $\phidot(t)$ in terms of their Fourier series,\footnote{We define the Fourier series expansion of a T-periodic function $f \in L^2(T)$ as $f(t) := \frac{1}{\sqrt{T}}\sum_{n \in \Zbb} \widehat{f}_ne^{\frac{2 \pi int}{T}},
\label{eq:def_FourierSeries}$ where $\widehat{f}_n: = \frac{1}{\sqrt{T}}\int_0^T f(t)e^{-\frac{2 \pi int}{T}}dt$.} we have that
\begin{equation}
\label{eq:eint_expand}
    \begin{aligned}
         \ebf_{\text{int}}(\phi, \ustar) & = \frac{\dt}{\sqrt{T}} \sum_{m=0}^{M}{}'' \parens{ \dot{u}^{\star}(t_m) \sum_{n \in \Zbb} \widehat{\phi}_n e^{\frac{2 \pi inm}{M}} }
         + \parens{ \ustar(t_m) \sum_{n \in \Zbb} \frac{2\pi i n}{T }\widehat{\phi}_n e^{\frac{2 \pi inm}{M}} } \\
         & = \frac{1}{\sqrt{T}} \sum_{n \in\Zbb}\widehat{\phi}_n \cdot \underbrace{\dt  \sum_{m=0}^{M}{}'' \frac{d}{dt} \brackets{\ustar e^{\frac{2 \pi int}{T}} }(t_m)}_{I_n}.
    \end{aligned}
\end{equation}
The asymptotic expansion of $I_n$ is derived by applying the Euler - Maclaurin formula \cite{DahlquistBjorck2008, Atkinson1989} to the integral $\int_0^T\frac{d}{dt} \brackets{\ustar e^{\frac{2 \pi int}{T}}}dt$. We recall the Euler-Maclaurin formula for the truncation error of the trapezoidal rule applied to the integral of a smooth function $g: [a, b] \rightarrow \mathbb{C}$ :
\begin{equation}
\begin{aligned}
     \dt \sum_{m = 0}^{M}{}'' g(t_m)  & \sim  \int_a^b g(t)dt   +  \sum_{s=1}^{\infty} \dt^{2s} \frac{B_{2s}}{(2s)!} \parens{ g^{(2s-1)}(b) - g^{(2s-1)}(a)},
\end{aligned}
\label{eq:EulerMaclaurin}
\end{equation}
where $B_{2k}$ are the Bernoulli numbers. Let $g(t)= \frac{d}{dt} \brackets{\ustar e^{\frac{2 \pi int}{T}}}$, it follows that $\int_0^T g(t)dt = \ustar(T) - \ustar(0)$. In addition, the $(2s-1)$-th derivative of $g(t)$ is given by
\begin{equation*}
    g^{(2s-1)}(t) = \parens{\ustar e^{\frac{2 \pi int}{T}}}^{(2s)} = \sum_{l = 0}^{2s} \binom{2s}{l} 
    \parens{\frac{2\pi in}{T} }^{2s-l}u^{\star(l)}(t)e^{\frac{2 \pi int}{T}},
\end{equation*}
and therefore $I_n$ admits the following approximation:
\begin{equation*}
    \begin{aligned}
        I_n & \approx  \ustar \big|_{0}^{T} + 
       \sum_{s=1}^{\infty} \dt^{2s} \frac{B_{2s}}{(2s)!} 
         \parens{ \sum_{l = 0}^{2s} \binom{2s}{l} 
         \parens{ \frac{2\pi in}{T} }^{2s-l}u^{\star(l)}\big|_{0}^{T}
        }.
    \end{aligned}
\end{equation*}
~
\end{proof}

\section{Multiscale Orthonormal Test Functions}
\label{sec:MinimumRadiusSelection}
In addition to introducing the $C^{\infty}$ bump function defined in \eqref{eq:tf_CinftyBump}, the work in  \cite{BortzMessengerDukic2023BullMathBiol} outlines a two-stage procedure for constructing matrices of multiscale globally supported orthonormal test functions, which we refer to as the \emph{Multiscale Global (MG)} approach. The first step involves identifying a minimum test function radius $r_{\text{min}}$ such that the integration error, as defined in \eqref{eq:e_int}, remains negligible compared to the noise-induced errors. 
For a $1$-dimensional $\ustar$, consider the $k$-th element of the integration error as defined in \eqref{eq:e_int}. Applying the product rule yields
\begin{equation}
\begin{aligned}
    \ebf_{\text{int}}(\phi_k, \ustar) = (\Gbfstar \wbf^{\star} - \bbf^{\star})_k &= \dt \sum_{m=0}^{M-1} \phi_k(t_m)\dot{u}^{\star}(t_m) + \phidot_k(t_m)\ustar(t_m)\\
    &= \frac{T}{M} \sum_{m=0}^{M-1}\frac{d}{dt}(\phi_k u^{\star})(t_m).
    \end{aligned}
    \label{eq:eint_k_prod}
\end{equation}
Rather than expressing $\phi_k$ in terms of its Fourier series as shown in \cref{lem:a}, the derivative $\frac{d}{dt}(\phi_k u^{\star})(t)$ is 
\begin{equation*}
\begin{aligned}
\frac{d}{dt}(\phi_k u^{\star})(t)
&= \frac{1}{\sqrt{T}}\sum_{n \in \Zbb} \frac{2 \pi i n}{T} \widehat{(\phi_k u^{\star})}_n e^{\frac{2 \pi i nt}{T}}.
\end{aligned}
\end{equation*}
Substituting this back into \eqref{eq:eint_k_prod} gives
\begin{equation*}
\begin{aligned}
        \ebf_{\text{int}}(\phi_k, \ustar, M) &= \frac{T}{M \sqrt{T}} \sum_{n \in \Zbb} \frac{2 \pi i n}{T} \widehat{(\phi_k u^{\star})}_n  \sum_{m=0}^{M-1} e^{\frac{2 \pi i n t_m}{T}}\\
        & =  \frac{2 \pi i }{\sqrt{T}} \sum_{n \in \Zbb} n  \widehat{(\phi_k u^{\star})}_{nM}.
\end{aligned}
\end{equation*}
This shows that the integration error depends solely on the $M, 2M, \dots, $ Fourier modes of $\phi_k u^{\star}$. Since the data is sampled at $M$ discrete time points, the number of available Fourier modes is limited to $ \lfloor \frac{M}{2}\rfloor$.

To estimate the integration error from discrete data, a subsampled grid is introduced, with spacing $\widetilde{\Delta t} = T / \lfloor M/s\rfloor$, where  $s>2$ is a coarsening factor. The approximate integration error over this coarse grid is
\begin{equation*}
    \widehat{\ebf}_\text{int}(u^\star,\phi_k,\lfloor M/s\rfloor,s) := \frac{2\pi i}{\sqrt{T}}\sum_{n=-\flr{s/2}}^{\flr{s/2}}n  \widehat{(\phi_k u^{\star})}_{n\lfloor M/s\rfloor}.
\end{equation*}
Since the Fourier coefficients decay as $M$ increases, we have
\begin{equation*}
    \widehat{\ebf}_\text{int}(u^\star,\phi_k,\lfloor M/s\rfloor,s) \approx \ebf_{\text{int}}(\phi_k, \ustar, \lfloor M/s\rfloor).
\end{equation*}
For coarsening factors between 2 and 4, the sum above reduces to just the $\pm  \lfloor M/s\rfloor$ modes
\begin{equation*}
\begin{aligned}
   \widehat{\ebf}_\text{int}(u^\star,\phi_k,\lfloor M/s\rfloor,s) &=   \frac{2\pi i}{\sqrt{T}} \big(
    - \widehat{(\phi_k u^{\star})}_{-\lfloor M/s\rfloor} +  \widehat{(\phi_k u^{\star})}_{\lfloor M/s\rfloor}
   \big) \\
   & = -\frac{4 \pi }{\sqrt{T}} \text{Im}\{ \widehat{(\phi_k u^{\star})}_{\lfloor M/s\rfloor}\},
\end{aligned}
\end{equation*}
where  $\text{Im}\{z\}$ denotes the imaginary part of a complex number $z$. 

Since access to the true noiseless trajectory $\ustar$ is unavailable,the integration error is computed using the noisy measurement $U$. For multivariate systems, this error is computed independently for each component. The minimum radius $\rmin$ is then selected as the changepoint of  the log of the root-mean-square (RMS) integration error
\begin{equation}
    \hat{\ebf}_\text{rms}(r):= \bigg(K^{-1}\sum_{k=1}^K\sum_{i=1}^{d}\widehat{\ebf}_\text{int}(\Ubf^{(i)},\phi_k(\cdot;r),\lfloor M/s\rfloor,s)^2 \bigg)^{\frac{1}{2}},
    \label{eq:IntErr}
\end{equation}
where $\Ubf^{(i)}$ denotes the $i$th variable. 

Once the minimal radius $r_{\text{min}}$ is determined, we construct the test function matrices $(\Phi,\dot{\Phi})$ by assembling and orthonormalizing test functions over multiple radii. Specifically, we define a set of increasing radii: 
 $\rbf:= \rmin \times(1,2,4,8)$. For each radius in 
$\rbf$, we build the corresponding test function matrices and vertically stack them:
\begin{equation*}
    \Phibf_{full} := \begin{bmatrix} 
    \Phibf_0^T   &
    \Phibf_1^T  &
    \Phibf_2^T &
    \Phibf_3^T \end{bmatrix}^T, 
    \quad   
    \Phidotbf_{full} := \begin{bmatrix} 
    \Phidotbf_0^T & 
    \Phidotbf_1^T & 
    \Phidotbf_2^T & 
    \Phidotbf_3^T 
    \end{bmatrix}^T.
\end{equation*}
%\begin{equation*}
    %\Phibf_{full} := \begin{bmatrix} \Phibf_0  \\  \Phibf_1 \\ \Phibf_2 \\ \Phibf_3 \end{bmatrix}, \quad   
%    \Phidotbf_{full} := \begin{bmatrix}  \Phidotbf_0  \\ \Phidotbf_1 \\ \Phidotbf_2 \\ \Phidotbf_3  \end{bmatrix}.
%\end{equation*}
Applying singular value decomposition to $\Phibf_{full}$ yields
$$\Phibf_{full} = \Qbf \mathbf{\Sigma} \Vbf^T.$$
Letting $k$ denote the truncation index (e.g., the SVD ``corner''), the final orthonormal test function matrices are computed by rescaling the projection onto the top $k$ left singular vectors
\begin{equation*}
    \begin{aligned}
      \Phibf = \text{diag}\Big(\frac{1}{\sigma_1}, \dots, \frac{1}{\sigma_k} \Big) \Qbf^T \Phibf_{full}, 
      \quad 
       \Phidotbf = \text{diag}\Big(\frac{1}{\sigma_1}, \dots, \frac{1}{\sigma_k} \Big) \Qbf^T \Phidotbf_{full} .
    \end{aligned}
\end{equation*}
The resulting matrices form an orthonormal basis of test functions that capture information across multiple scales \cite{RummelMessengerBeckerEtAl2025arXiv250208881}. To conclude this section, we summarize the full procedure as the \emph{Multi-scale-global (MG)} approach in \cref{alg:MG}. A detailed exposition of this method is provided in \cite{BortzMessengerDukic2023BullMathBiol}.

\begin{algorithm}
\caption{Multiscale Global Approach (MG) (see \cref{sec:MinimumRadiusSelection})}
\label{alg:MG}
\begin{algorithmic}
\STATE \textbf{Input:} Data $\Ubf$, coarsening factor $s$
\STATE \textbf{Output:} Test function matrices $\Phibf$, $\Phidotbf$
\vspace{0.5em}
\STATE \textit{/* Compute minimum radius $\rmin$ */}
\STATE For each candidate radius $r = m\dt$, with $m = 2, \cdots, \lfloor M/2 \rfloor$:
\begin{itemize}
    \item Construct a candidate basis of test functions $\{\phi_k(r)\}_{k=1}^K$
    \item For each component of the data $\Ubf_i$, compute  $\widehat{\ebf}_\text{int}^{(i,k)}(r)$ using:
    $$
    \widehat{\ebf}_\text{int}^{(i,k)}(r) := -\frac{4 \pi}{\sqrt{T}} \, \mathrm{Im}\left\{ \widehat{(\phi_k(r) \Ubf_i)}_{\lfloor M/s\rfloor} \right\}.
    $$
    \item Compute the root mean square of $\widehat{\ebf}_\text{int}^{(i,k)}(r)$ across all variables and test functions:
    $$
    \widehat{\ebf}_{\text{rms}}(r) := \sqrt{ \frac{1}{K} \sum_{i=1}^{d} \sum_{k=1}^{K} \left( \widehat{\ebf}_\text{int}^{(i,k)}(r) \right)^2 }.
    $$
\end{itemize}
\STATE Identify $\rmin := \text{changepoint}\parens{\log(\widehat{\ebf}_{\text{rms}}(r))}$ using \cref{alg:changepoint}.
\vspace{0.5em}
\STATE \textit{/* Compute Multiscale Global Basis */}
\STATE Construct test function matrices:
$$\Phibf_{\text{full}} :=\begin{bmatrix}\Phibf^T(\rmin) &\Phibf^T(2\rmin) & \Phibf^T(4\rmin) & \Phibf^T(8\rmin)
\end{bmatrix}^T, $$ 
$$\Phidotbf_{\text{full}} := 
\begin{bmatrix}
\Phidotbf^T(\rmin) & \Phidotbf^T(2\rmin) & \Phidotbf^T(4\rmin) & \Phidotbf^T(8\rmin)
\end{bmatrix}^T. $$
\STATE Perform SVD: $\Phibf_{\text{full}} = \Qbf \boldsymbol{\Sigma} \Vbf^T$.
\STATE Select truncation index $k$, construct: 
$\Pbf: = \text{diag}\left(\sigma_1^{-1}, \dots, \sigma_k^{-1} \right) \Qbf^T$.
\STATE Compute $\Phibf := \Pbf \Qbf^T \Phibf_{\text{full}} \quad \Phidotbf := \Pbf \Qbf^T \Phidotbf_{\text{full}}$
\RETURN $\Phibf, \Phidotbf$.
\end{algorithmic}
\end{algorithm}

\section{Proof of \cref{prop:eint_hat}}
\label{sec:proof_prop_ehat}
\begin{proof}
The expansion in \cref{lem:a} characterizes the integration error for a single test function $\phi$, and thus can be applied to each $\phi_k$ to produce an expression for each component of the vector $\ebfint$. From \cref{def:tf_set}, we have $\phi_k(t) = \psi(t - t_{m_k})$, so the Fourier coefficients of $\phi_k$ satisfy $\widehat{\phi_k}_n = e^{-\frac{2\pi inm_k}{M}}\widehat{\psi}_n$. Therefore, the $k$-th component of $\ebfint$ is
\begin{equation*}
       (\ebf_{\text{int}})_k := \ebfint(\phi_k, \ustar) = \frac{1}{\sqrt{T}} \sum_{n\in \Zbb} \widehat{\phi_k}_n I_n = \frac{1}{\sqrt{T}} \sum_{n \in \mathbb{Z}} e^{-\frac{2\pi inm_k}{M}} \widehat{\psi}_n I_n.
    \end{equation*}
We truncate this infinite sum to retain only the $M$ Fourier modes in the set of frequencies $\nbf$ defined in \eqref{eq:freq}. For a discussion of truncating beyond $M$ modes, see \cref{sec:Mtilde}. The resulting approximation of the integration error for $\phi_k$ is
\begin{equation*}
    \ebfhatint(\phi_k, \ustar) = \frac{1}{\sqrt{T}} \sum_{n \in \nbf} e^{-\frac{2\pi inm_k}{M} }\widehat{\psi}_n I_n, 
\end{equation*}
Note that the vector $\left( e^{-\frac{2\pi i n m_k}{M}} \right)_{n \in \nbf}$ corresponds to the $m_k$-th row of the DFT matrix $\Fbf$.

The approximation of each integral $I_n$ is computed  by truncating the Euler– Maclaurin expansion from \cref{lem:a}, with truncation order $S$ and finite difference order $\boldsymbol{\mu}$: 
\begin{equation*} 
    I_n(\ustar, S, \boldsymbol{\mu}) := \ustar|_0^T + \sum_{s=1}^{S} \dt^{2s} \frac{B_{2s}}{(2s)!} \parens{ \sum_{l = 0}^{2s} \binom{2s}{l} \parens{\frac{2 \pi i n}{T}}^{2s-l} \mathcal{D}^{l}_{\mu_l}[\ustar]|_0^T }.
\end{equation*}
Concatenating these approximations over all $n \in \nbf$ gives the vector $\Ibf(\ustar, S, \boldsymbol{\mu})$. Finally, stacking $\ebfhatint(\phi_k, \ustar)$ over $ k = 1, \dots, K$ yields \eqref{eq:eint_hat}, completing the proof.
\end{proof}

\section{Proof of \cref{prop:psi_hat}} 
\label{sec:psi_hat}
\begin{proof}
The reference test function is given by: 
$$\psi(t;r) = 
\begin{cases}
     C(r^2 - t^2)^p \quad & t \in [-r, r]\\
     \quad  \quad 0 \quad & \text{otherwise}
\end{cases}
$$
Since  $\psi(t;r)$ is an even function, its Fourier coefficients satisfies $\widehat{\psi}_n = \widehat{\psi}_{-n}$. For the case $n = 0$, a simple change of variable $x = \frac{t}{r}$ yields
\begin{equation*}
    \begin{aligned}
    \widehat{\psi}_0 & = \frac{2C}{\sqrt{T}} \int_{0}^{r} (r^2 - t^2)^p dt = \frac{2C}{\sqrt{T}} r^{2p+1} \int_{0}^{1}(1-x^2)^p dx \\
    & = \frac{2C}{\sqrt{T}} r^{2p+1} \sum_{j=0}^p \binom{p}{j}\frac{(-1)^j}{2j+1}.
    \end{aligned}
\end{equation*}
For $n \geq 1$, we have that
\begin{equation*}
    \begin{aligned}
    \widehat{\psi}_n &=   \frac{C}{\sqrt{T}} \int_{-r}^{r} (r^2 - t^2)^p e^{-\frac{2 \pi int}{T}}dt \\
    & = \frac{C}{\sqrt{T}} \int_{-r}^{r} (r^2 - t^2)^p cos\bigg(\frac{2\pi nt}{T} \bigg)dt 
    - \frac{iC}{\sqrt{T}} \int_{-r}^{r} (r^2 - t^2)^p sin\bigg(\frac{2\pi nt}{T} \bigg)dt \\
    & = \frac{C}{\sqrt{T}} \sqrt{\pi} \bigg(\frac{rT}{n\pi} \bigg)^{p+\frac{1}{2}} \Gamma(p+1) J_{p + \frac{1}{2}}\bigg( \frac{2\pi nr}{T}\bigg),
    \end{aligned}
\end{equation*}
because $\psi$ is even, which means the integral against the sin term vanishes. The remaining integral involving the cosine term follows from formula 3.771.8 in \cite{GradshteynRyzhik2015TableofIntegralsSeriesandProducts}
$$\int_0^u (u^2 - x^2)^{\nu - \frac{1}{2}}cos(ax)d =\frac{\sqrt{\pi}}{2} \Big(\frac{2u}{a} \Big)^{\nu}\Gamma\Big(\nu + \frac{1}{2}\Big)J_{\nu}(au),$$
where $\Gamma$ is the  Gamma function and $J_{\nu}$ is the Bessel function of the first kind.
\end{proof}
The normalization $C$ is given by $ C = \frac{1}{\norm{\psi}_{2}}$, where 
$$\norm{\psi}_2 = r^{2p} \sqrt{2r}\sqrt{\sum_{k=0}^{2p}\binom{2p}{k}\frac{(-1)^k}{2k+1}}.$$

%Different ways to normalize $L^2$ and $L^{\infty}$, i.e., $C = \frac{1}{\norm{\psi}_{2}}$ or $C = \frac{1}{\norm{\psi}_{\infty}}$, where:
%begin{equation*}
%\begin{aligned}
    %\norm{\psi}_2 = r^{2p} \sqrt{2r}\sqrt{\sum_{k=0}^{2p}\binom{2p}{k}\frac{(-1)^k}{2k+1}}, \quad \quad \quad 
   % \norm{\psi}_{\infty} = r^{2p}
%\end{aligned}
%\end{equation*}

\section{Choosing the truncation order $S$}
\label{sec:apdxS}
In this section, we provide a general guideline for selecting the truncation order $S$ in the asymptotic expansion of $I_n$ in Eq. \eqref{eq:I_approx}. Ideally, this choice depends on the smoothness of $u$. If such information is available, $S$ should be chosen accordingly to match the regularity of $u$. 

In the absence of explicit smoothness information, a practical strategy is to incrementally increase $S$ until the contribution of the $S$-th term becomes negligible relative to the cumulative contribution up to order $S-1$. Specifically, we select $S$ such that the relative difference satisfies
\begin{equation*}
    \frac{|I_n(S) - I_n(S-1)|}{ |I_n(S-1)|}  < \tau,
\end{equation*}
for some tolerance $\tau$ and $S > 1$. This translates to 
\begin{equation*}
    \begin{aligned}
    \frac{\bigg|\Delta t^{2S} \frac{B_{2S}}{(2S)!} 
          \Big( \sum_{l = 0}^{2S} \binom{2S}{l} 
         \big(\frac{2\pi in}{T} \big)^{2S-l}u^{(l)}\big|_{0}^{T}
         \Big)\bigg|} { \big| I_n(S-1) \big|} < \tau.
    \end{aligned}
\end{equation*}
In most practical scenarios, it is sufficient to estimate the dominant contributions without computing additional higher-order derivatives. For instance, we may approximate the truncation condition as
\begin{equation*}
    \begin{aligned}
      \frac{\bigg|\Delta t^{2S} \frac{B_{2S}}{(2S)!}  
         \big(\frac{2\pi in}{T} \big)^{2S}u\big|_{0}^{T}
         \bigg|} { \big| I_n(S-1) \big|} < \tau,
    \end{aligned}
\end{equation*}
which avoids the need to compute derivatives beyond order $2(S-1)$. In practice, we set the tolerance to $\tau = 0.1$, and in all numerical examples presented in this work, a truncation order of $S=1$ was found to be sufficient.

\section{Approximation of the integration error with oversampled Fourier modes}
\label{sec:Mtilde}
According to \cref{lem:a}, the integration error $\ebf_{\text{int}}(\phi_k, \ustar)$ can be expressed as an infinite sum, i.e., a linear combination of the Fourier coefficients of the reference test function $\psi$. The closed-form expression for these coefficients provided in \cref{prop:psi_hat} enables analytical computation of $\widehat{\psi}_n$, allowing us to construct a surrogate approximation $\widehat{\ebf}_{\text{int}}$ beyond the usual resolution of $M$ modes. Specifically, we extend this estimate to incorporate $\widetilde{M} > M$ Fourier modes, as defined in the modified version of \eqref{eq:eint_hat} as
\begin{equation}
 \label{eq:eq:eint_approx_tilde}
    \begin{aligned}
         \widehat{\ebf}_{\text{int}}(\ustar, r, S, \widetilde{M}, \boldsymbol{\mu}) := \Pbf \Fbf \Psihatbf(r, \widetilde{M}) \Ibf(\ustar, S, \widetilde{M}, \boldsymbol{\mu}) \in \mathbb{R}^K.
    \end{aligned}
\end{equation}
Here, $\Ibf(\ustar, S, \widetilde{M}, \boldsymbol{\mu}) \in \mathbb{C}^{\widetilde{M}}$ is the approximation of the integrals $I_n$ \eqref{eq:eint_expand},
%the Fourier integrals $\frac{d}{dt} \brackets{\ustar e^{\frac{2 \pi int}{T}} }$, 
defined analogous to \eqref{eq:I_approx} using $\widetilde{M}$ frequency modes \eqref{eq:freq}. Similarly, the diagonal matrix $\Psihatbf \in \mathcal{C}^{\widetilde{M} \times \widetilde{M}}$ contains the Fourier coefficients $\widehat{\psi}_n$ up to frequency $\widetilde{M}/2$. The DFT matrix $\Fbf \in \mathcal{C}^{M \times \widetilde{M}}$ becomes rectangular, with entries given by $\Fbf_{m, n}:= e^{-2 \pi inm/M}$. The projection matrix $\Pbf$ remains as defined in \eqref{eq:P_proj}, selecting the time indices corresponding to the test function centers. When $\widetilde{M} = \alpha M$ for some integer $\alpha > 1$, the matrix product $\Pbf \Fbf \Psihatbf \Ibf$ can be computed efficiently using the Fast Fourier Transform. 

\begin{figure}[htbp]
  \centering
  \begin{subfigure}[b]{0.49\textwidth}
    \includegraphics[width=\textwidth]{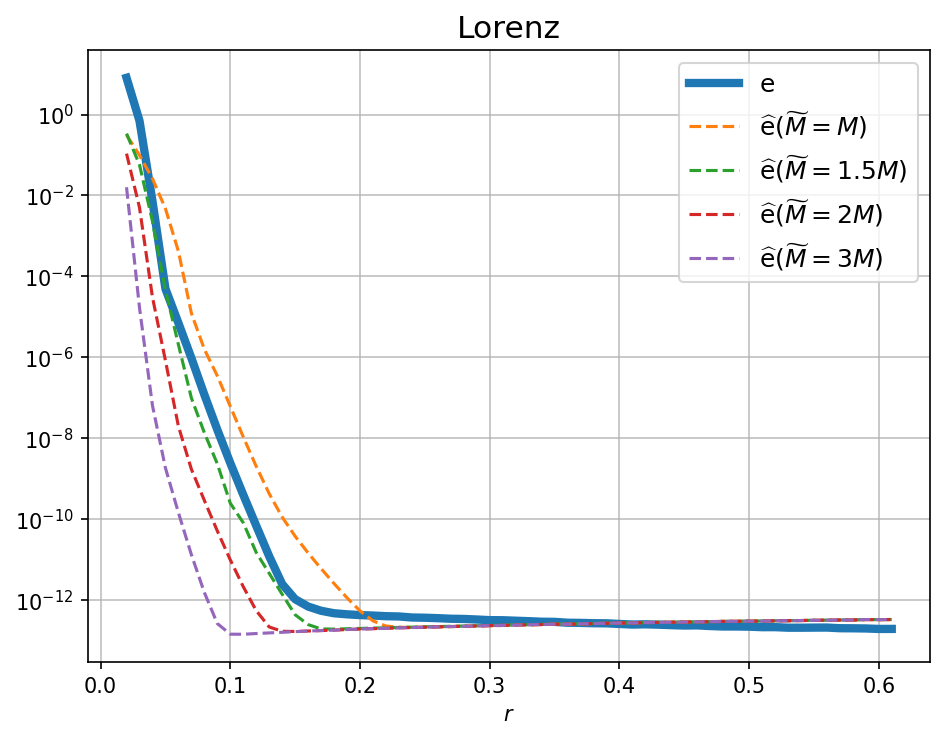}
  \end{subfigure}
  \hfill
  \begin{subfigure}[b]{0.49\textwidth}
    \includegraphics[width=\textwidth]{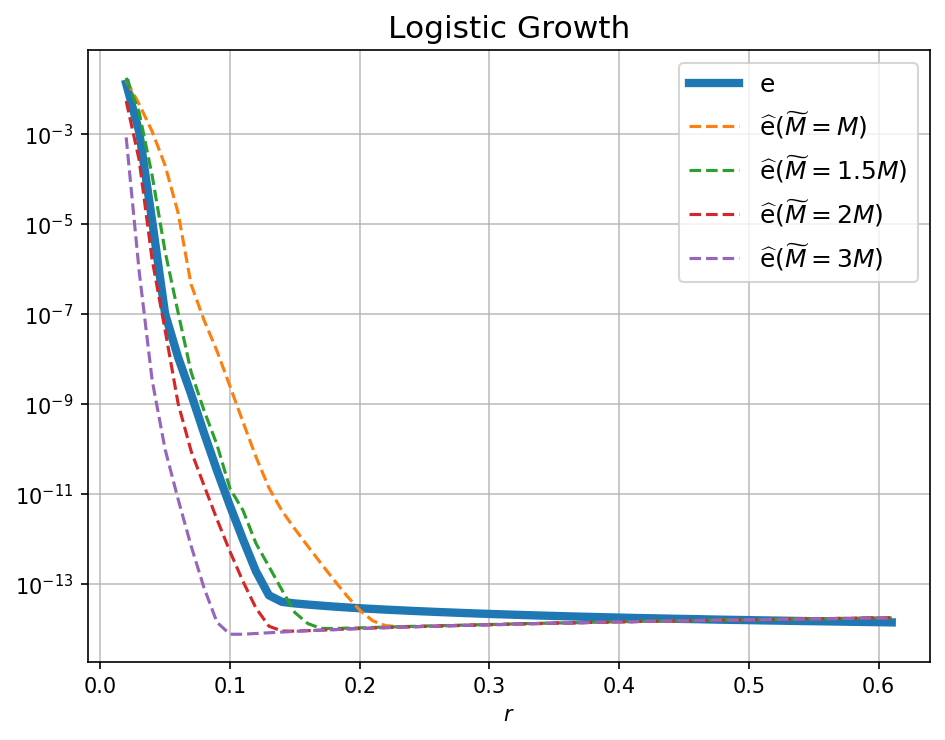}
  \end{subfigure}
  \caption{Comparison of the true integration error $\erm$ and the estimated error $\widehat{\erm}$ for the Lorenz system (left) and logistic growth model (right), plotted on a log scale. Results are shown for noise-free data with test function order $p = 16$. Estimates $\widehat{\erm}$ are computed with $\widetilde{M} \in {M, 1.5M, 2M, 3M}$. Solid blue lines represent the true integration error $\erm$. Notably, increasing $\Mtilde$ does not lead to a closer approximation of $\erm$.}
  \label{fig:eehat_Mtilde}
\end{figure}

%Computing $\widehat{\ebf}_{\text{int}}$ and $\widehat{\ebf}$ as defined in \eqref{eq:e_int_hat} and \eqref{eq:e_hat}, we obtain the following results for $\Mtilde = M, 1.5M$ and $2M$ 

\Cref{fig:eehat_Mtilde} illustrates that increasing $\widetilde{M}$ does not cause the estimated error $\widehat{\erm}$ to converge exactly to the true error $\erm$, even in the absence of noise. This discrepancy likely arises for two reasons: (1) the integral approximation $\Ibf$ is based on a truncated series, which cannot fully capture the exact integral without access to the true parameters $\wbf^\star$; and (2) as the radius $r$ increases, the high-frequency Fourier modes decay rapidly, and thus rounding errors may dominate.

As shown in \cref{fig:eehat_Mtilde}, varying $\widetilde{M}$ among $M$, $1.5M, 2M$, and $3M$ shows that while larger values of $\widetilde{M}$ may more closely track the decay of $\erm$, the change point detected by $\widehat{\erm}$ when $\widetilde{M} = M$ already falls within the low-error region. Based on this observation, and considering the computational advantages discussed in \cref{remark:FFT}, we fix $\widetilde{M} = M$ throughout the proposed algorithm. Exploring other values of $\widetilde{M}$ is left for future work.

\end{document}